\documentclass[a4paper,11pt]{article}

\usepackage[sep]{optional}
\usepackage{ucs}
\usepackage[utf8x]{inputenc}
\usepackage{type1ec}
\usepackage[utf8x]{inputenc}

\usepackage{hyperref}
\usepackage[frenchb]{babel}
\usepackage{fullpage}

\usepackage{bbm}
\usepackage[bbgreekl]{mathbbol}
\usepackage{nccrules}
\usepackage{tocbibind}
\usepackage[intlimits,leqno]{amsmath}
\usepackage{amsthm}
\usepackage{amssymb}
\usepackage{mathrsfs}
\usepackage{stmaryrd}
\usepackage{xspace}
\usepackage[all]{xy}
\newdir{ >}{{}*!/-10pt/\dir{>}}

\newcommand{\Z}{\ensuremath{\mathbb{Z}}}

\newcommand{\R}{\ensuremath{\mathbb{R}}}
\newcommand{\C}{\ensuremath{\mathbb{C}}}

\newcommand{\A}{\ensuremath{\mathbb{A}}}

\newcommand{\Tr}{\ensuremath{\mathrm{tr}\,}}

\newcommand{\Resprod}{\ensuremath{{\prod}'}}

\newcommand{\dd}{\ensuremath{\,\mathrm{d}}}

\newcommand{\angles}[1]{\ensuremath{\langle #1 \rangle}}
\newcommand{\mes}{\ensuremath{\mathrm{mes}}}

\newcommand{\identity}{\ensuremath{\mathrm{id}}}

\newcommand{\Hom}{\ensuremath{\mathrm{Hom}}}

\newcommand{\rightiso}{\ensuremath{\stackrel{\sim}{\rightarrow}}}

\newcommand{\Ker}{\ensuremath{\mathrm{Ker}\,}}

\theoremstyle{plain}
\newtheorem{proposition}{Proposition}[section]
\newtheorem{lemma}[proposition]{Lemme}
\newtheorem{theorem}[proposition]{Théorème}
\newtheorem{corollary}[proposition]{Corollaire}

\theoremstyle{definition}
\newtheorem{definition}[proposition]{Définition}
\newtheorem{definition-theorem}[proposition]{Définition-Théorème}
\newtheorem{definition-proposition}[proposition]{Définition-Proposition}

\newtheorem{remark}[proposition]{Remarque}

\newcommand{\bmu}{\ensuremath{\bbmu}}

\newcommand{\noyau}{\ensuremath{\boldsymbol{\varepsilon}}} % l'élément non trivial dans le noyau du revêtement métaplectique
\newcommand{\rev}{\ensuremath{\mathbf{p}}} % Le symbole pour revêtements
\renewcommand{\Re}{\ensuremath{\mathrm{Re}\xspace}}
\renewcommand{\Im}{\ensuremath{\mathrm{Im}\xspace}}

\title{La formule des traces pour les revêtements de groupes réductifs connexes. III. \\ Le développement spectral fin}
\author{Wen-Wei Li}
\date{}

\begin{document}

\maketitle

%\selectlanguage{english}
%\begin{abstract}
%  foobar
%\end{abstract}

\begin{abstract}
  Nous poursuivons l'étude de la formule des traces pour certains revêtements de groupes réductifs connexes en déduisant une formule explicite du côté spectral, basée sur des résultats d'analyse harmonique locale dans les travaux précédents. Les arguments sont dus à Arthur et nous expliquons comment ils s'adaptent aux revêtements.
\end{abstract}

%MSC 2010: 11F72 (Primary), 11F70 (Secondary)
\begin{flushleft}
  \small MSC classification (2010): \textbf{11F72}, 11F70.
\end{flushleft}

\tableofcontents

% To be included
% La formule des traces pour les revêtements de groupes réductifs connexes. III. Le développement spectral fin
\opt{these}{\chapter{Le développement spectral fin}}

\section{Introduction}
Cet article succède à \cite{Li10a, Li11a}, qui visent à établir la formule des traces invariante d'Arthur pour une certaine classe d'extensions centrales des groupes topologiques
$$ 1 \to \bmu_m \to \tilde{G} \xrightarrow{\rev} G(\A) \to 1 $$
où $\bmu_m := \{\noyau \in \C^\times : \noyau^m=1\}$, $G$ est un groupe réductif connexe sur un corps de nombres $F$, et $\A = \A_F$ est l'anneau d'adèles. On exige aussi certaines propriétés rendant l'étude des formes automorphes possible, par exemple un scindage de $\rev$ au-dessus de $G(F)$ et la commutativité des algèbres de Hecke sphériques en presque toute place non archimédienne de $F$; voir \cite{Li10a} pour les détails. On l'appelle un revêtement adélique à $m$ feuillets.

La formule des traces grossière, ainsi qu'elle est énoncé dans \cite{Li10a}, est une égalité de la forme
$$ \underbrace{\sum_{\mathfrak{o} \in \mathcal{O}^G} J_{\mathfrak{o}}(f)}_{\text{le côté géométrique}} = J(f) = \underbrace{\sum_{\chi \in \mathfrak{X}^{\tilde{G}}} J_\chi(f)}_{\text{le côté spectral}}, \quad f \in C^\infty_c(\tilde{G}^1). $$

Ici $\mathcal{O}^G$ signifie l'ensemble des classes de conjugaison semi-simples dans $G(F)$ et $\mathfrak{X}^{\tilde{G}}$ signifie l'ensemble des données cuspidales automorphes, autrement dit des $W^G_0$-orbites des paires $(M,\sigma)$ où $M$ est un sous-groupe de Lévi semi-standard et $\sigma$ est une représentation automorphe cuspidale de $\tilde{M}^1$.

Les termes sont effectivement les valeurs en $T=T_0$ des polynômes $J^T_{\mathfrak{o}}(f)$, $J^T_\chi(f)$ et $J^T(f)$ en $T \in \mathfrak{a}_0$, où $T_0$ est un point canonique défini par Arthur dépendant du choix d'un sous-groupe compact maximal. Le terme $J^T(f)$ est l'intégrale diagonale d'un certain noyau tronqué de l'action régulière de $f$ sur $L^2(G(F) \backslash \tilde{G}^1)$ pour $T$ suffisamment régulier. Une telle formule n'est pas assez utile pour des problèmes globaux. Par exemple, $J^T_\chi(f)$ est défini en termes des produits scalaires des séries d'Eisenstein tronquées, intégrés sur l'axe imaginaire, dont on ne s'attend à aucune formule utile et générale. Il faut donc trouver des formules plus fines pour les termes $J_{\mathfrak{o}}(f)$ et $J_{\chi}(f)$.

Le côté géométrique est déjà traité dans \cite{Li10a}, où est obtenu un développement à la Arthur \cite{Ar86} en termes des intégrales orbitales pondérées. Dans cet article, nous considérons le côté spectral en suivant toujours les arguments d'Arthur \cite{Ar82-Eis1, Ar82-Eis2}. Nous utilisons les fonctions test $\tilde{K}$-finies afin d'avoir des résultats de finitude, où $\tilde{K}$ est un sous-groupe compact maximal convenable de $\tilde{G}$. Pour de telles fonctions $f$ nous obtenons une formule dans le Corollaire \ref{prop:formule-J_chi-simple}:
\begin{multline*}
  J_\chi(f) = \sum_{M \in \mathcal{L}(M_0)} \sum_{\pi \in \Pi_\mathrm{unit}(\tilde{M}^1)} \sum_{L \in \mathcal{L}(M)} \sum_{s \in W^L(M)_\text{reg}} |W^M_0| |W^G_0|^{-1} \cdot \\
  \cdot |\det(s-1|\mathfrak{a}^L_M)|^{-1} \int_{i(\mathfrak{a}^G_L)^*} \Tr(\mathcal{M}_L(\tilde{P},\lambda) M_{P|P}(s,0) \mathcal{I}_{\tilde{P}}(\lambda,f)_{\chi,\pi}) \dd\lambda
\end{multline*}
où $P \in \mathcal{P}(M)$ est arbitraire.

Suivant \cite{Ar88-TF2}, nous regroupons ensuite les $\chi$ selon les normes des parties imaginaires de leurs caractères infinitésimaux $\nu_\chi$. On introduit ainsi $J_t(f) := \sum_{\chi: \|\Im \nu_\chi\|=t} J_\chi(f)$ pour tout $t \geq 0$. Si l'on extrait les termes avec $L=G$ et $\|\Im\nu_\chi\|=t$ dans la formule précédente, on arrive à la partie $t$-discrète $J_{\text{disc},t}(f)$ de la formule des traces. C'est la partie qui nous intéresse. Grâce à des propriétés d'admissibilité enterrées dans la construction du spectre discret \cite[V]{MW94}, on introduit un ensemble $\Pi_{\text{disc},t}(\tilde{G})$ de représentations de $\tilde{G}^1$ et des coefficients $a^{\tilde{G}}_{\text{disc}}(\pi)$ pour tout $\pi \in \Pi_{\text{disc},t}(\tilde{G})$, tel que
$$ J_{\text{disc},t}(f) = \sum_{\pi \in \Pi_{\text{disc},t}(\tilde{G})} a^{\tilde{G}}_{\text{disc}}(\pi) \Tr\pi(f). $$

Ce n'est pas clair si $\sum_t J_{\text{disc},t}(f)$ converge absolument. Nous n'abordons pas cette question malgré son utilité dans des applications. Toutefois il n'est pas impossible d'adapter les méthodes dans le cas réductif \cite{FLM11}. 

Remarquons que, dans \cite{Li10a} on ne considère que la partie spécifique de la formule des traces, c'est-à-dire on se limite aux représentations spécifiques et aux fonctions anti-spécifiques, bien que le cas général en résulte sans peine. Cette restriction est cruciale pour le côté géométrique mais n'importe pas pour le côté spectral. On autorise ainsi toutes les représentations et fonctions test de $\tilde{G}^1$ dans cet article.

Selon \cite{Ar02}, l'étape suivante sera de définir les caractères pondérés locaux canoniquement normalisés \cite{Ar98}, et les coefficients globaux $a^{\tilde{M}}(\cdot)$ pour tout sous-groupe de Lévi semi-standard $M$. On obtiendra ainsi un développement de $J_t(f)$ en termes de ces objets, pour tout $t \geq 0$. Ceci diffère de l'approche antérieure dans \cite{Ar88-TF2} en ce que nos caractères pondérés seront canoniquement définis, et que nous utiliserons les facteurs normalisants locaux non ramifiés pour les coefficients $a^{\tilde{M}}(\cdot)$ qui font intervenir des fonctions $L$ partielles au lieu des fonctions $L$ automorphes et des facteurs $\varepsilon$ selon l'approche antérieure. Pour cela, il faut aussi une majoration de la convergence absolue en termes de multiplicateurs \cite[\S 6]{Ar88-TF2}. Nous les laissons à un article à venir.

Un mot sur les définitions: nous dévions des conventions d'Arthur adoptées dans \cite[\S 6.1]{Li10a} en changeant la définition de l'espace $\mathcal{A}^2(\tilde{P})$ (voir \S\ref{sec:prod-Eis}). La nouvelle convention a deux avantages. D'abord les termes $\rho_P$ disparaissent dans les définitions des séries d'Eisenstein et des opérateurs d'entrelacement globaux. Deuxièmement, cela est plus compatible avec la situation locale dans \cite{Li11a} et facilite la comparaison local-global des opérateurs d'entrelacement dans \S\ref{sec:conv}. Ce changement n'affecte nullement les arguments d'Arthur.

Enfin, donnons quelques remarques concernant le style et les preuves. Les arguments sont tous dus à Arthur et un expert peut s'en convaincre immédiatement. Plus précisément, les ingrédients importants sont le suivants.
\begin{enumerate}
  \item Décomposition spectrale \cite{MW94}, notamment la construction du spectre discret par résidus. C'est valable pour les revêtements.
  \item Majorations précises pour le côté spectral de la formule des traces grossière \cite[Appendix]{Ar82-Eis1}. Cette partie est de nature plus élémentaire modulo la décomposition spectrale, et nous ne le traitons pas en profondeur.
  \item Théorème de Paley-Wiener d'Arthur \cite{Ar83}, qui est valable pour tout groupe réel dans la classe de Harish-Chandra, y compris les composantes archimédiennes des revêtements adéliques.
  \item Formule du produit scalaire des séries d'Eisenstein tronquées \cite{Ar82-IP}. Malheureusement cela n'est pas encore complètement rédigée, même dans \cite{MW94}. Nous nous contenterons d'en donner les énoncés et d'indiquer qu'elle s'appuie sur la théorie de base de la décomposition spectrale, eg. le calcul des termes constants des séries d'Eisenstein.
  \item Normalisation des opérateurs d'entrelacement locaux, qui est faite dans \cite{Li11a} pour les revêtements.
  \item Arguments combinatoires d'Arthur, eg. la théorie de $(G,M)$-familles, qui marchent sans modification.
\end{enumerate}

Ayant justifiés ces ingrédients, notre tâche est de clarifier la logique des démonstrations d'Arthur. Parfois nous ne donnons qu'une esquisse de la démonstration afin d'éviter de créer un texte géant et illisible. D'autre part, nous remplissons aussi les détails de certains arguments d'Arthur, par exemple dans la Proposition \ref{prop:adm}.

\paragraph{Organisation de cet article}
Dans le \S\ref{sec:prem}, nous rappelons le formalisme pour les revêtements mis dans \cite{Li10a}.

Dans le \S\ref{sec:prod-Eis}, nous mettons en place les définitions des séries d'Eisenstein et des opérateurs d'entrelacement. Nous donnons ensuite l'énoncé complet de la formule du produit scalaire des séries d'Eisenstein tronquées, qui permet de donner une description simple asymptotique pour $J_\chi^T(f)$.

Dans le \S\ref{sec:asymptotique}, nous utilisons un théorème de multiplicateur d'Arthur pour insérer une fonction $B \in C_c^\infty(i\mathfrak{h}^*/i\mathfrak{a}_G^*)^W$ avec $B(0)=1$ dans l'intégrale définissant $J_\chi^T(f)$ et définissons un polynôme $P^T(B)$ asymptotique à $J^T_\chi(f)$, qui permet de contourner des problèmes analytiques. On obtient $J_\chi^T(f)$ en évaluant $\lim_{\epsilon \to 0} P^T(B^\epsilon)$ avec $B^\epsilon(\cdot) = B(\epsilon \cdot)$.

On arrive à une formule explicite pour $P^T(B)$ dans le \S\ref{sec:conv} en introduisant des $(G,M)$-familles. La famille $c_Q$ contrôle la dépendance sur $T$, tandis que l'autre famille $d_Q$ est définie à l'aide des opérateurs d'entrelacement. La formule cherchée résulte alors des formules de descente dans la théorie de $(G,M)$-familles.

Pour se débarrasser des termes $B^\epsilon$ dans les formules précédentes, il faut montrer que des fonctions $r^S_L(\pi_\lambda, P)$ qui se déduisent de facteurs normalisants globaux sont à croissance modérée en $\lambda \in i(\mathfrak{a}^G_L)^*$, ce qui est le sujet du \S\ref{sec:conv}. Enfin, on définit $J_{\text{disc},t}(f)$ et déduit un développement dans le \S\ref{sec:coef}.

\paragraph{Remerciements}
Cet article est rédigé à la fin de ma Longue Marche doctorale à Paris. J'exprime ma profonde gratitude à mon directeur de thèse Jean-Loup Waldspurger, qui a très attentivement lu le manuscrit, ainsi que les membres de l'équipe des formes automorphes de l'Institut de Mathématiques de Jussieu pour les excellentes conditions de travail qu'ils ont fournies pendant ces trois années.

\section{Préliminaires}\label{sec:prem}
Rappelons brièvement le formalisme dans \cite{Li10a} pour les revêtements. On se donne
\begin{itemize}
  \item $F$: un corps de nombres;
  \item soit $v$ une place de $F$, on note $F_v$ le complété de $F$ en $v$; si $v$ est non archimédienne, on note $\mathfrak{o}_v$ son anneau des entiers;
  \item $\A = \A_F := \Resprod_v F_v$: l'anneau d'adèles de $F$;
  \item $G$: un $F$-groupe réductif connexe;
  \item une extension centrale $1 \to \bmu_m \to \tilde{G} \stackrel{\rev}{\to} G(\A) \to 1$ de groupes topologiques, où $m$ est un entier positif, qui forme un revêtement adélique à $m$ feuillets; en particulier on dispose d'un scindage $\mathbf{i}: G(F) \to \tilde{G}$ de $\rev$, par lequel $G(F)$ s'identifie à un sous-groupe discret de $\tilde{G}$ (voir \cite[\S 3]{Li10a});
  \item $M_0$: un sous-groupe de Lévi minimal de $G$.
\end{itemize}

\paragraph{Notations combinatoires}
Nous utilisons les notations dans \cite{Li10a} concernant les sous-groupes de Lévi, les paraboliques et les racines, etc, qui sont pour l'essentiel celles d'Arthur. Par exemple, on note $W^G_0$ le groupe de Weyl associé à $M_0$. Si $M \in \mathcal{L}(M_0)$, on note
\begin{align*}
  W^G(M) & := \{w \in W^G_0: wMw^{-1} = M\}/W^M_0 = \text{Norm}_G(M)/M, \\
  W^G(M)_\text{reg} & := \{w \in W^G(M): \Ker(w-1|\mathfrak{a}^G_M) = \{0\} \}.
\end{align*}

Soient $M, M' \in \mathcal{L}(M)$, on pose
$$ W(\mathfrak{a}_M, \mathfrak{a}_{M'}) := \{w \in W^G_0: wMw^{-1}=M' \} / W^M_0 . $$

Fixons un sous-groupe compact maximal spécial $K = \prod_v K_v$ de $G(\A)$ tel que $K_v$ est en bonne position relativement à $M_0$ pour toute place $v$. Il convient aussi de fixer un sous-groupe parabolique minimal $P_0 \in \mathcal{L}(M_0)$, bien que nos résultats finaux n'en dépendront pas. Arthur montre dans \cite{Ar81} qu'il existe un unique élément $T_0 \in \mathfrak{a}^G_0$ tel que
$$ H_{P_0}(\hat{w}^{-1}) = (1-w^{-1})T_0, \quad w \in W^G_0 $$
où $\hat{w}$ désigne un représentant quelconque de $w$ dans $G(F)$. On peut écrire $\hat{w} = m\tilde{w}$ avec $m \in M_0(\A)$ et $\tilde{w} \in K$, d'où $H_{P_0}(\hat{w}^{-1})=H_{M_0}(m^{-1})$. Donc $T_0$ ne dépend pas du choix de $P_0 \in \mathcal{P}(M_0)$.

\paragraph{Revêtements}
Si $v$ est une place de $F$, on note $\rev_v : \tilde{G} \to G(F_v)$ la fibre de $\rev$ au-dessus de $G(F_v)$. Dans la définition des revêtements adéliques, c'est sous-entendu que l'on fixe un ensemble fini $V_\text{ram}$ des places de $F$ contenant toutes les places archimédiennes. On note $\mathfrak{o}_\text{ram}$ l'anneau des $V_\text{ram}$-entiers dans $F$ et on fixe un $\mathfrak{o}_\text{ram}$-modèle du schéma en groupes $G$. Pour $v \notin V_\text{ram}$, on prend $K_v := G(\mathfrak{o}_v)$. Grosso modo, l'hypothèse est qu'il existe un scindage $s_v: K_v \hookrightarrow \tilde{G}_v$ de $\rev_v$ tel que le triplet $(\rev_v, K_v, s_v)$ fournit un revêtement non ramifié pour tout $v \notin V_\text{ram}$, et que les applications $s_v$ définissent une section du revêtement adélique sur un ouvert compact (cf. \cite[\S 3.1 et \S 3.3]{Li10a}). Quitte à agrandir $V_\text{ram}$, on peut supposer de plus que $K_v$ est en bonne position relativement à $M_0$ pour toute place $v \notin V_\text{ram}$. Fixons un sous-groupe compact maximal $K = \prod_v K_v$ de $G(\A)$ et posons $\tilde{K} := \rev^{-1}(K)$, tel que
\begin{itemize}
  \item $K_v$ est spécial et en bonne position relativement à $M_0$ pour tout $v$;
  \item pour tout $v \notin V_\text{ram}$, $K_v$ est le groupe hyperspécial choisi précédemment qui s'identifie à un sous-groupe de $\tilde{G}_v$ à l'aide du scindage $s_v$.
\end{itemize}

Les éléments et sous-groupes de $\tilde{G}$ sont dotés du symbole $\sim$, par exemple $\tilde{x}$, $\tilde{P}$, $\tilde{M}$, etc, tandis que leurs images par $\rev$ sont notées par $x$, $P$, $M$, etc.

On a $\tilde{G} = \Resprod_v \tilde{G}_v / \left\{ (\noyau_v)_v \in \bigoplus_v \bmu_m : \prod_v \noyau_v = 1 \right\}$, où le produit restreint est pris par rapport à $(K_v)_{v \notin V_\text{ram}}$.

Rappelons aussi que, pour toute place $v$, il existe un scindage canonique de $\rev_v$ au-dessus de la sous-variété unipotente $G_\text{unip}(F_v)$, qui se restreint en un homomorphisme sur chaque sous-groupe unipotent. Ces scindages se rassemblent en un scindage du revêtement adélique $\rev$. Identifions ainsi les éléments unipotents de $G(\A)$ à des éléments de $\tilde{G}$.

Imposons les mêmes conventions que dans \cite{Li10a} pour les mesures de Haar sur les revêtements ainsi que les espaces vectoriels $\mathfrak{a}_M$, $\mathfrak{a}^G_M$, etc. Fixons une norme $W^G_0$-invariante $\|\cdot\|$ sur $\mathfrak{a}_0$ induisant la mesure de Haar choisie sur $\mathfrak{a}_0$.

\paragraph{Représentations et fonctions}
Soit $H$ un groupe localement compact, on note $\Pi_\mathrm{unit}(H)$ l'ensemble de classes d'équivalences des représentations unitaires irréductibles de $H$. Soit $\pi \in \Pi_\mathrm{unit}(\tilde{G})$. Il n'est pas toujours un produit tensoriel restreint car $\tilde{G}$ n'est pas forcément un produit restreint. Or on peut tirer $\pi$ en une représentation de $\Resprod_v \tilde{G}_v$ et puis l'écrire comme un produit tensoriel restreint, grâce à notre hypothèse sur la commutativité des algèbres de Hecke sphériques \cite[\S 3.2]{Li10a}: la clé est que pour toute $v \notin V_\text{ram}$ et toute représentation lisse irréductible de $\tilde{G}_v$, son sous-espace fixé par $K_v$ est de dimension $1$. Ainsi, par abus de notation on écrira
$$ \pi = \bigotimes_v \pi_v . $$

Pour $\pi$ comme ci-dessus et $\lambda \in \mathfrak{a}_G^*$, on déduit une nouvelle représentation irréductible
$$ \pi_\lambda := \pi \otimes e^{\angles{\lambda, H_{\tilde{G}}(\cdot)}}. $$

Soient $M \in \mathcal{L}^G(M_0)$, $\pi \in \Pi_\mathrm{unit}(\tilde{M})$ et $w \in W^G_0$ avec un représentant $\hat{w} \in G(F)$. On note $w\pi \in \Pi_\mathrm{unit}(w\tilde{M}w^{-1})$ la représentation $(w\pi)(\tilde{m}') = \pi(\hat{w}^{-1}\tilde{m}' \hat{w})$ pour $\tilde{m}' \in w\tilde{M}w^{-1}$. La représentation est indépendante du choix de $\hat{w}$ à équivalence près. Idem pour $\tilde{M}^1$ au lieu de $\tilde{M}$. Rappelons d'ailleurs que nous avons défini dans \cite[3.4]{Li10a} un sous-groupe central $A_{M,\infty}$ de $\tilde{M}$ qui est isomorphe à un espace euclidien (et désolé pour le conflit potentiel de notations), tel que  $\tilde{M} = \tilde{M}^1 \times A_{M,\infty}$. Ainsi, $\Pi_\mathrm{unit}(\tilde{M}^1)$ se plonge dans $\Pi_\mathrm{unit}(\tilde{M})$.

Soit $\phi: G(F) \backslash \tilde{G}^1 \to \C$ une fonction localement intégrable. Soit $P \in \mathcal{F}(M_0)$, on note $\phi_P(\tilde{x}) := \int_{U(F) \backslash U(\A)} \phi(u\tilde{x}) \dd u$ son terme constant le long de $\tilde{P}$. C'est encore une fonction localement intégrable.

Soit $P \in \mathcal{F}(M_0)$, posons
$$ d_P(T) := \sup \{ \angles{\alpha, T} : \alpha \in \Delta_P \}, \quad T \in \mathfrak{a}_0 $$
et écrivons $d_0(T) := d_{P_0}(T)$.

Pour $T \in \mathfrak{a}_0$ avec $d_0(T) \gg 0$, on peut définir la fonction tronquée $\Lambda^T \phi$ par
$$ \Lambda^T \phi(\tilde{x}) = \sum_{P=MU \supset P_0} (-1)^{\dim A_M/A_G} \sum_{\gamma \in P(F) \backslash G(F)} \hat{\tau}_P(H_P(\gamma x)-T) \phi_{\tilde{P}}(\gamma\tilde{x}) $$
où $\hat{\tau}_P$ est la fonction caractéristique de
$$ \{ H \in \mathfrak{a}_0 : \forall \alpha \in \Delta_P, \; \angles{\varpi_\alpha, H} > 0 \}. $$

\section{Produit scalaire des séries d'Eisenstein tronquées}\label{sec:prod-Eis}
\paragraph{Définitions de base}
Fixons les objets $\rev: \tilde{G} \to G(\A)$, $P_0$, $M_0$, $K$ comme dans la section précédente. Soient $M \in \mathcal{L}(M_0)$, $P=MU \in \mathcal{P}(M)$ tel que $P \supset P_0$. L'espace $L^2(M(F) \backslash \tilde{M}^1)$ est un espace de Hilbert de façon évidente. On définit les espaces suivants
\begin{align*}
  \mathcal{A}(\tilde{P}) = \mathcal{A}(U(\A)M(F) \backslash \tilde{G}) & : \quad \text{les formes automorphes sur } U(\A)M(F) \backslash \tilde{G} , \\
  \mathcal{A}_\text{cusp}(\tilde{P}) & : \quad \text{le sous-espace des formes cuspidales}.
\end{align*}
Cf. \cite[I.2.17]{MW94}; dans cet article nous factorisons le sous-groupe $A_{M,\infty}$, par conséquent nous ne fixons pas le caractère central.

Pour $\phi \in \mathcal{A}(\tilde{P})$ et $\tilde{x} \in \tilde{G}$, on définit $\phi_{\tilde{x}}(\cdot) := \delta_P(\cdot)^{-\frac{1}{2}} \phi(\cdot \tilde{x})$ qui est une fonction sur $\tilde{M}$. Définissons ensuite

\begin{equation*}
  \mathcal{A}^2(\tilde{P}) := \left\{ \begin{aligned}
    \phi \in \mathcal{A}(\tilde{P}) : & \forall \tilde{x} \in \tilde{G}, \phi_{\tilde{x}} \in L^2(M(F) A_{M,\infty} \backslash \tilde{M}),  \\
    & \text{ et } \int_{\tilde{K}} \int_{M(F) \backslash \tilde{M}^1} |\phi_{\tilde{k}}(\tilde{m})|^2 \dd\tilde{m}\dd\tilde{k} < +\infty
  \end{aligned} \right\},
\end{equation*}
$$ \mathcal{A}^2_\text{cusp}(\tilde{P}) := \mathcal{A}^2(\tilde{P}) \cap \mathcal{A}_\text{cusp}(\tilde{P}). $$

Désormais nous identifions $M(F) A_{M,\infty} \backslash \tilde{M}$ et $M(F) \backslash \tilde{M}^1$. Soient $\phi, \phi' \in \mathcal{A}^2(\tilde{P})$, on pose
$$ (\phi|\phi') := \int_{\tilde{K}} \int_{M(F) \backslash \tilde{M}^1} \phi_{\tilde{k}}(\tilde{m}) \overline{\phi'_{\tilde{k}}(\tilde{m})} \dd\tilde{m} \dd\tilde{k} $$
qui fournit une structure pré-hilbertienne. On note $\overline{\mathcal{A}^2}(\tilde{P})$, etc, les complétés hilbertiens ainsi obtenus. 

Soit $\lambda \in \mathfrak{a}_{M,\C}^*$, on définit une représentation de $\tilde{G}$ sur $\overline{\mathcal{A}^2}(\tilde{P})$, notée $\mathcal{I}_{\tilde{P}}(\lambda)$, en posant
$$ (\mathcal{I}_{\tilde{P}}(\lambda, \tilde{y})\phi)(\tilde{x}) = \phi(\tilde{x}\tilde{y}) e^{\angles{\lambda , H_P(xy)-H_P(x)}}, \quad \tilde{x},\tilde{y} \in \tilde{G}. $$

Si $\pi \in \Pi_\mathrm{unit}(\tilde{M}^1)$, on note $L^2(M(F) \backslash \tilde{M}^1)_\pi$ le sous-espace $\pi$-isotypique de $L^2(M(F) \backslash \tilde{M}^1)$. Notons $\mathcal{A}^2(\tilde{P})_\pi$ l'espace des fonctions $\phi \in \mathcal{A}^2(\tilde{P})$ telles que $\phi_{\tilde{x}} \in L^2(M(F) \backslash \tilde{M}^1)_\pi$ pour tout $\tilde{x}$. On définit $\mathcal{A}^2_\text{cusp}(\tilde{P})_\pi$ de la même façon. Remarquons que $\mathcal{A}^2(\tilde{P})_\pi \neq \{0\}$ si et seulement si $\pi$ intervient dans $L^2_\text{disc}(M(F) \backslash \tilde{M}^1)$ d'après \cite[V.3.17]{MW94}. Donc $\mathcal{I}_{\tilde{P}}(\lambda)$ s'identifie à l'induite parabolique normalisée de $L^2_\text{disc}(M(F) \backslash \tilde{M}^1) \otimes e^{\angles{\lambda, H_M(\cdot)}}$.

Soient $\phi \in \mathcal{A}(\tilde{P})$ et $\lambda \in \mathfrak{a}_{M,\C}^*$, on définit la série d'Eisenstein $E(\phi, \lambda)$, qui est la fonction $\tilde{x} \mapsto E(\tilde{x},\phi,\lambda)$ sur $G(F) \backslash \tilde{G}$ donnée par la formule suivante lorsque $\angles{\Re\lambda,\alpha^\vee} \gg 0$ pour tout $\alpha \in \Delta_P$

$$ E(\tilde{x}, \phi, \lambda) = \sum_{\gamma \in P(F) \backslash G(F)} \phi(\gamma\tilde{x}) e^{\angles{\lambda, H_{P_1}(\gamma x)}}. $$

Cela se prolonge en une fonction méromorphe en $\lambda \in \mathfrak{a}_{M,\C}^*$. D'autre part, soient $\lambda$ comme ci-dessus, $w \in W^G_0$ et $M' := wMw^{-1}$ tels qu'il existe $P' = M'U' \in \mathcal{P}(M')$ avec $P' \supset P_0$, on a l'opérateur d'entrelacement
$$ M(w, \lambda):  \mathcal{A}(\tilde{P}) \to \mathcal{A}(\tilde{P}'). $$
Lorsque $\angles{\Re\lambda,\alpha^\vee} \gg 0$ pour tout $\alpha \in \Delta_P$, cela est défini par l'intégrale absolument convergente
 
$$ M(w,\lambda)\phi: \tilde{x} \longmapsto \int_{(U' \cap wUw^{-1})(\A) \backslash U'(\A)} \phi(\hat{w}^{-1} u\tilde{x}) e^{\angles{\lambda, H_P(\hat{w}^{-1}ux)}} e^{\angles{-w\lambda, H_{P'}(x)}} \dd u $$
où $\hat{w} \in G(F)$ est un représentant quelconque de $w$. Cela se prolonge en une fonction méromorphe en $\lambda$. Il entrelace $\mathcal{I}_{\tilde{P}}(\lambda)$ et $\mathcal{I}_{\tilde{P}'}(w\lambda)$, et il envoie $\mathcal{A}^2(\tilde{P})_\pi$ sur $\mathcal{A}^2(\tilde{P}')_{w\pi}$. On vérifie aussi qu'il ne dépend que de la classe $w W^M_0$. On a les équations fonctionnelles suivantes
\begin{align*}
  E(M(w,\lambda)\phi, w\lambda) & = E(\phi,\lambda), \\
  M(w_1 w_2, \lambda) & = M(w_1, w_2\lambda) M(w_2, \lambda).
\end{align*}

C'est connu que $E(\phi,\lambda)$ et $M(w,\lambda)\phi$ sont réguliers lorsque $\lambda \in i\mathfrak{a}_M^*$ et $\phi \in \mathcal{A}^2(\tilde{P})$; de plus, $M(w,\lambda)$ est unitaire sur $\overline{\mathcal{A}^2}(\tilde{P})$. Voir \cite[VI.2]{MW94}.

L'ensemble des données automorphes cuspidales $\mathfrak{X} = \mathfrak{X}^{\tilde{G}}$ est formé des $W^G_0$-orbites des paires $(M',\sigma)$ où $M' \in \mathcal{L}(M_0)$ et $\sigma$ est une représentation automorphe cuspidale de $(\tilde{M}')^1$. Il y a une application naturelle $\mathfrak{X}^{\tilde{M}} \to \mathfrak{X}^{\tilde{G}}$ à fibres finies. La théorie de la décomposition spectrale donne une décomposition orthogonale
$$ L^2(M(F) \backslash \tilde{M}^1) = \bigoplus_{\chi \in \mathfrak{X}} L^2(M(F) \backslash \tilde{M}^1)_\chi . $$
Précisons. Soit $\chi^M=[M',\sigma] \in \mathfrak{X}^{\tilde{M}}$. Grosso modo, $L^2(M(F) \backslash \tilde{M}^1)_{\chi^M}$ est le sous-espace obtenu en prenant les résidus des séries d'Eisenstein associées à $\sigma$, cf. \cite[VI]{MW94}. On définit
$$ L^2(M(F) \backslash \tilde{M}^1)_\chi := \bigoplus_{\chi^M \mapsto \chi} L^2(M(F) \backslash \tilde{M}^1)_{\chi^M}. $$

On définit ainsi les espaces
\begin{align*}
  \mathcal{A}^2(\tilde{P})_\chi & := \{ \phi \in \mathcal{A}^2(\tilde{P}) : \forall \tilde{x} \in \tilde{G}, \phi_{\tilde{x}} \in L^2(M(F) \backslash \tilde{M}^1)_\chi \}, \\
  \mathcal{A}^2(\tilde{P})_{\chi,\pi} & := \mathcal{A}^2(\tilde{P})_\pi \cap \mathcal{A}^2(\tilde{P})_{\chi}.
\end{align*}

En prenant les complétés hilbertiens dans $\overline{\mathcal{A}^2}(\tilde{P})$, on définit les espaces $\overline{\mathcal{A}^2}(\tilde{P})_\pi$, $\overline{\mathcal{A}^2}(\tilde{P})_{\chi,\pi}$, etc. Ce sont des sous-espaces invariants de $\mathcal{I}_{\tilde{P}}(\lambda)$ pour tout $\lambda \in \mathfrak{a}_{M,\C}^*$. On note les sous-représentations ainsi obtenues par $\mathcal{I}_{\tilde{P}}(\lambda)_\pi$, $\mathcal{I}_{\tilde{P}}(\lambda)_{\chi, \pi}$, etc. Enfin, il convient parfois de fixer un sous-ensemble fini $\Gamma \subset \Pi_\mathrm{unit}(\tilde{K})$, i.e. des $\tilde{K}$-types, et on introduit les sous-espaces $\mathcal{A}^2(\tilde{P})_{\chi, \pi, \Gamma}$, etc., qui sont engendrés par les vecteurs transformant selon les éléments de $\Gamma$. Ce sont des espaces vectoriels de dimension finie.

\paragraph{La formule du produit scalaire}
Pour $Y \in \mathfrak{a}_0$, notons comme d'habitude $Y_G$ la projection de $Y$ sur $\mathfrak{a}_G$. Introduisons la notation suivante
\begin{align*}
  \tilde{G}^Y & := \{\tilde{x} \in \tilde{G} : H_{\tilde{G}}(\tilde{x})=Y_G \}, \\
  (h|h')_{\tilde{G},Y} & := \int_{G(F) \backslash \tilde{G}^Y } h(\tilde{x}) \overline{h'(\tilde{x})} \dd \tilde{x}, \quad h,h' \in L^2(G(F) \backslash \tilde{G}^Y), \\
  (h|h')_{\tilde{G}} & := (h|h')_{\tilde{G},0}, \quad h,h' \in L^2(G(F) \backslash \tilde{G}^1).
\end{align*}

Soient $S,T \in \mathfrak{a}_0$ avec $d_0(T) \gg 0$. D'après une propriété générale des opérateurs de troncature¸ la série d'Eisenstein tronquée $\Lambda^T E(\phi, \lambda)$ restreinte à $G(F) \backslash \tilde{G}^S$ est de carré intégrable pour tout $\phi \in \mathcal{A}^2(\tilde{P})$. 

\begin{theorem}[cf. {\cite[\S 8]{Ar82-IP}}]\label{prop:prod-Eis-1}
  Fixons les données $\chi, \Gamma$ comme ci-dessus. Soit $P'=M'U'$ un autre sous-groupe parabolique standard de $G$. Alors il existe
  \begin{itemize}
    \item un sous-ensemble $W(\mathfrak{a}_M, \mathfrak{a}_{M'}, \chi)$ de $W(\mathfrak{a}_M, \mathfrak{a}_{M'})$;
    \item un sous-ensemble fini $\mathcal{E}\text{xp}$ de $\mathfrak{a}_0^*$;
  \end{itemize}
  tels que
  \begin{itemize}
    \item $\angles{X, \varpi_\alpha^\vee} \leq 0$ pour tout $X \in \mathcal{E}\text{xp}$ et tout $\alpha \in \Delta_0$;
    \item $0 \in \mathcal{E}\text{xp}$
  \end{itemize}
  et les propriétés suivantes soient vérifiées: soient $\phi \in \mathcal{A}^2(\tilde{P})_{\chi, \Gamma}$, $\phi' \in \mathcal{A}^2(\tilde{P}')_{\chi, \Gamma}$, $\lambda \in \mathfrak{a}_{M,\C}^*$ et $\lambda' \in \mathfrak{a}_{M',\C}^*$, on a un développement
  $$ (\Lambda^T E(\phi, \lambda) | \Lambda^T E(\phi', -\bar{\lambda}') )_{\tilde{G},T} = \sum_{X \in \mathcal{E}\text{xp}} q_X^{T,\tilde{G}}(\lambda, \lambda', \phi, \phi') e^{\angles{X,T}} $$
  avec
  $$ q_X^{T,\tilde{G}}(\lambda, \lambda', \phi, \phi') = \sum_{(t,t') \in W(\mathfrak{a}_M, \mathfrak{a}_{M'}, \chi)} p_{X,t,t'}^{T,\tilde{G}}(\lambda,\lambda',\phi,\phi') e^{\angles{t\lambda-t'\lambda', T}} $$
  où $p_{X,t,t'}^{T,\tilde{G}}(\lambda,\lambda', \phi, \phi')$ est un polynôme en $T$. En tant qu'une fonction en $(\lambda,\lambda')$, $p_{X,t,t'}^{T,\tilde{G}}(\cdot,\cdot, \phi,\phi')$ est méromorphe, elle est régulière lorsque $\lambda \in i\mathfrak{a}_M^*$, $\lambda' \in i\mathfrak{a}_{M'}^*$.
\end{theorem}

Vu les propriétés de $\mathcal{E}\text{xp}$, le comportement asymptotique de $(\Lambda^T E(\phi, \lambda) | \Lambda^T E(\phi', -\bar{\lambda}') )_{\tilde{G},T}$ en $T$ est contrôlé par le terme $q_0^{T,\tilde{G}}(\lambda,\lambda',\phi,\phi')$. Une description beaucoup plus précise est donnée ci-dessous. Posons
\begin{equation}
  \omega^T(\lambda, \lambda',\phi,\phi') := \sum_{\substack{P_1 \supset P_0 \\ P_1 = M_1 U_1}} \sum_{\substack{t \in W(\mathfrak{a}_M|\mathfrak{a}_{M_1}) \\ t' \in W(\mathfrak{a}_{M'}|\mathfrak{a}_{M_1})}} (M(t,\lambda)\phi | M(t', -\overline{\lambda'})\phi') \frac{e^{\angles{t\lambda-t'\lambda', T}}}{\theta_{P_1}(t\lambda - t'\lambda')}.
\end{equation}
Voir \cite[\S 4.2]{Li10a} pour les définitions de $\theta_{P_1}$. Fixons aussi $\delta, N > 0$ avec $N$ suffisamment grand, et posons
\begin{align}
  \mathcal{T} := \{T \in \mathfrak{a}_0^+ : d_0(T) > \delta\|T\| > N \}.
\end{align}

\begin{theorem}[{\cite[Theorem 9.1]{Ar82-IP}}]\label{prop:prod-Eis-2}
  Conservons les notations précédentes.
  \begin{enumerate}
    \item On a
      $$ q_0^{T,\tilde{G}}(\lambda,\lambda',\phi,\phi') = \omega^T(\lambda, \lambda',\phi,\phi').$$
    \item Il existe $\epsilon > 0$ et une fonction localement bornée $\rho: i\mathfrak{a}_M^* \times i\mathfrak{a}_{M'}^* \to \R_{\geq 0}$, tels que
      $$ |(\Lambda^T E(\phi, \lambda) | \Lambda^T E(\phi', \lambda') )_{\tilde{G},T} - \omega^T(\lambda,\lambda',\phi,\phi')| \leq \rho(\lambda,\lambda') \|\phi\| \|\phi'\| e^{-\epsilon \|T\|} $$
      pour $\phi$, $\phi'$ comme dans le Théorème \ref{prop:prod-Eis-1}, $T \in \mathcal{T}$ et $\lambda \in i\mathfrak{a}_M^*$, $\lambda' \in i\mathfrak{a}_{M'}^*$.
  \end{enumerate}
\end{theorem}

\begin{corollary}\label{prop:prod-Eis-3}
  Il existe $\epsilon > 0$ et une fonction localement bornée $\rho: i\mathfrak{a}_M^* \to \R_{\geq 0}$ tels que
  $$ |(\Lambda^T E(\phi, \lambda) | \Lambda^T E(\phi', \lambda) )_{\tilde{G}} - \omega^T(\lambda,\lambda,\phi,\phi')| \leq \rho(\lambda) \|\phi\| \|\phi'\| e^{-\epsilon \|T\|} $$
  pour $\phi$, $\phi'$ comme dans le Théorème \ref{prop:prod-Eis-1}, $T \in \mathcal{T}$ et $\lambda \in i\mathfrak{a}_M^*$.
\end{corollary}

\begin{proof}[Ingrédients pour les démonstrations]
  Les outils essentiels pour ces résultats principaux de \cite{Ar82-IP} sont
  \begin{enumerate}\renewcommand{\labelenumi}{(\roman{enumi})}
    \item la théorie de base des séries d'Eisenstein, leurs termes constants, les opérateurs d'entrelacement, etc;
    \item le cas $\phi \in \mathcal{A}^2_\text{cusp}(\tilde{P})_{\chi,\Gamma}$, $\phi' \in \mathcal{A}^2_\text{cusp}(\tilde{P}')_{\chi,\Gamma}$ traité par Langlands, cf. \cite[Lemma 4.2]{Ar80} et \cite[\S 9]{Lan66};
    \item la construction du spectre discret par résidus.
  \end{enumerate}\renewcommand{\labelenumi}{(\arabic{enumi})}

  En fait, les parties profondes (i) et (iii) pour les revêtements sont traitées dans \cite{MW94}, tandis que (ii) est plus élémentaire et s'appuie sur (i), pour l'essentiel.
\end{proof}

\section{Étude asymptotique du côté spectral}\label{sec:asymptotique}
\paragraph{Majorations}
Dans cette section, on fixe $\chi \in \mathfrak{X}^{\tilde{G}}$. Le symbole $T$ désignera toujours un élément dans $\mathfrak{a}_0$. Dans \cite{Li10a}, nous avons défini la distribution $f \mapsto J^T_\chi(f)$ où $f \in C_c^\infty(\tilde{G})$. C'est un polynôme en $T$.

Fixons $P=MU$ et $\pi \in \Pi_\mathrm{unit}(\tilde{M})$. En supposant que $d_0(T) \gg 0$, on définit un opérateur
$$ \Omega^T_{\chi,\pi}(\tilde{P},\lambda): \mathcal{A}^2(\tilde{P})_{\chi,\pi} \to \mathcal{A}^2(\tilde{P})_{\chi,\pi}, \quad \lambda \in i\mathfrak{a}_M^* $$
par la formule suivante
$$ (\Omega^T_{\chi,\pi}(\tilde{P},\lambda)\phi | \phi') = (\Lambda^T E(\phi,\lambda) | \Lambda^T E(\phi',\lambda))_{\tilde{G}}, \quad \phi,\phi' \in \mathcal{A}^2(\tilde{P})_{\chi,\pi}. $$

On a aussi démontré au cours d'établir la formule des traces grossière que $\Omega^T_{\chi,\pi}(\tilde{P},\lambda) \mathcal{I}_{\tilde{P}}(\lambda,f)_{\chi,\pi}$ est un opérateur à trace (voir la Définition \ref{def:trace}). C'est donc loisible de poser
\begin{equation}\label{eqn:Psi}
  \Psi^T_{\chi,\pi}(\lambda,f) := |\mathcal{P}(M)|^{-1} \Tr(\Omega^T_{\chi,\pi}(\tilde{P},\lambda) \mathcal{I}_{\tilde{P}}(\lambda,f)_{\chi,\pi}).
\end{equation}
On l'écrira parfois $\Psi^T_{\pi}(\lambda,f)$. La fonction $\Psi^T_{\chi,\pi}(\lambda,f)$ n'est pas déterminée par $\pi|_{\tilde{M}^1}$, cependant son intégrale sur $\lambda \in i(\mathfrak{a}^G_M)^*$ l'est; on sait que la somme sur $\pi$ de ces intégrales est égale à $J_\chi^T(f)$ lorsque $d_0(T) \gg 0$. Nous devons préciser les phrases ``$d_0(T) \gg 0$'' dans la suite.

\begin{proposition}[cf. {\cite[Proposition 2.1]{Ar82-Eis1}}]\label{prop:majoration-1}
  Il existe des entiers positifs $C_0, d_0$ tels que pour tous $f \in C_c^\infty(\tilde{G}^1)$, $n \geq 0$ et $T \in \mathfrak{a}_0$ avec $d_0(T) > C_0$, il existe une constante $c_{n,f}$ indépendante de $T$ vérifiant
  $$ \sum_{P=MU \supset P_0} \sum_{\pi \in \Pi_\mathrm{unit}(\tilde{M}^1)} \int_{i(\mathfrak{a}^G_M)^*} |\Psi^T_{\chi,\pi}(\lambda,f)| (1+\|\lambda\|)^n \dd\lambda \leq c_{n,f} (1+\|T\|)^{d_0}. $$
\end{proposition}

Choisissons une fonction hauteur $\|\cdot\|: \tilde{G} \to \R_{\geq 0}$ comme dans \cite{Li10a}. Soit $N \in \R$, posons
$$ C_N^\infty(\tilde{G}) := \{ f \in C_c^\infty(\tilde{G}) : f(\tilde{x})=0 \text{ si } \log\|\tilde{x}\| > N \}. $$
Idem pour $C_N^\infty(\tilde{G}^1)$.

\begin{proposition}[cf. {\cite[Proposition 2.2]{Ar82-Eis1}}]\label{prop:majoration-2}
  Il existe une constante $C_0 > 0$ telle que pour tous $N > 0$, $f \in C_N^\infty(\tilde{G}^1)$, $T \in \mathfrak{a}_0$ tel que $d_0(T) > C_0(1+N)$, on a
  $$ J^T_\chi(f) = \sum_{P=MU \supset P_0} \sum_{\pi \in \Pi_\mathrm{unit}(\tilde{M}^1)}\; \int_{i(\mathfrak{a}^G_M)^*} \Psi^T_{\chi,\pi}(\lambda,f) \dd\lambda . $$
  En particulier, le côté à droite est un polynôme en $T$.
\end{proposition}
Nous employons la même lettre $C_0$ pour la constante car il n'y aura aucune confusion à craindre.
\begin{proof}
  Il suffit de reprendre les arguments dans \cite[Appendix]{Ar82-Eis1}, qui sont élémentaires modulo la décomposition spectrale.
\end{proof}

Introduisons maintenant les espaces
\begin{align*}
  \mathcal{H}(\tilde{G}) & := \{f \in C_c^\infty(\tilde{G}) : f \text{ est } \tilde{K}-\text{finie} \}, \\
  \mathcal{H}(\tilde{G}^1) & := \{f \in C_c^\infty(\tilde{G}^1) : f \text{ est } \tilde{K}-\text{finie} \}, \\
  \mathcal{H}_N(\tilde{G}) & := \mathcal{H}(\tilde{G}) \cap C_N^\infty(\tilde{G}), \\
  \mathcal{H}_N(\tilde{G}^1) & := \mathcal{H}(\tilde{G}^1) \cap C_N^\infty(\tilde{G}^1).
\end{align*}
Plus précisément, on peut fixer $\Gamma$ un sous-ensemble fini de $\Pi_\mathrm{unit}(\tilde{K})$ et noter $\mathcal{H}(\tilde{G}^1)_\Gamma$ (resp. $\mathcal{H}(\tilde{G})_\Gamma$) l'espace des $f \in \mathcal{H}(\tilde{G}^1)$ (resp. $f \in \mathcal{H}_\Gamma(\tilde{G})$) dont les $\tilde{K}$-types sont contenus dans $\Gamma$. Alors $\mathcal{H}(\tilde{G}^1) = \bigcup_{\Gamma} \mathcal{H}(\tilde{G}^1)_\Gamma$ (resp. $\mathcal{H}(\tilde{G}) = \bigcup_{\Gamma} \mathcal{H}(\tilde{G})_\Gamma$). Idem pour les composantes locales du revêtement. Le résultat suivant garantit que toutes les sommes sur $\pi$ que nous considérons dans la suite sont finies.

\begin{proposition}\label{prop:finitude-pi}
  Soient $\Gamma$ un sous-ensemble fini de $\Pi_\mathrm{unit}(\tilde{K})$ et $f \in \mathcal{H}(\tilde{G}^1)_\Gamma$. Il n'existe qu'un nombre fini de $\pi \in \Pi_\mathrm{unit}(\tilde{M}^1)$ tels que
  \begin{itemize}
    \item $\pi$ contient des restriction à $\tilde{K} \cap \tilde{M}$ des éléments de $\Gamma$,
    \item $\mathcal{A}^2(\tilde{P})_{\chi,\pi} \neq \{0\}$.
  \end{itemize}
  En particulier, la somme dans la Proposition \ref{prop:majoration-2} est finie.
\end{proposition}
\begin{proof}
  C'est une conséquence de la construction du spectre discret par résidus: voir \cite[VI]{MW94}, notamment l'assertion sur l'admissiblité des paramètres discrets et le Corollaire VI.1.8.
\end{proof}
Désormais, nous prenons toujours $f \in \mathcal{H}(\tilde{G}^1)$.

\paragraph{Application d'un théorème de multiplicateurs}
Notons $F_\infty = \prod_{v|\infty} F_v$ et
$$ \tilde{G}_\infty := \rev^{-1}(G(F_\infty)),$$
qui est un groupe de Lie dans la classe de Harish-Chandra. Les mêmes notations s'appliquent aux sous-groupes de Lévi de $G$. Alors $\tilde{K}_\infty := \rev^{-1}(\prod_{v|\infty} K_v)$ est un sous-groupe compact maximal de $\tilde{G}_\infty$ en bonne position relativement à $\widetilde{M_{0,\infty}}$. Notons $\mathfrak{g}_\infty$ l'algèbre de Lie de $\tilde{G}_\infty$ et $\mathfrak{g}_{\infty,\C} := \mathfrak{g}_\infty \otimes_\R \C$. On prend
\begin{itemize}
  \item $\mathfrak{h}_K$: une sous-algèbre de Cartan de l'algèbre de Lie de $\tilde{K}_\infty \cap \widetilde{M_{0,\infty}}$, et
  \item $\mathfrak{h}_0$: l'algèbre de Lie d'un tore réel déployé maximal dans $\widetilde{M_{0,\infty}}$
\end{itemize}
de sorte que
$$ \mathfrak{h} := i\mathfrak{h}_K \oplus \mathfrak{h}_0 \subset \mathfrak{g}_{\infty,\C} $$
est une sous-algèbre de Cartan de la forme déployée de $\mathfrak{g}_\infty$. Notons $\mathfrak{h}_\C := \mathfrak{h} \otimes_\R \C$ et
$$ W := W(\mathfrak{g}_{\infty,\C},\mathfrak{h}_\C) $$
le groupe de Weyl absolu. Alors $\mathfrak{h}$ est invariant par $W$. Pour tout $P \in \mathcal{F}(M_0)$, on définit
\begin{align*}
  h_P: \mathfrak{h} & \twoheadrightarrow \mathfrak{a}_P, \\
  X + Y & \mapsto H_P(\exp Y), \quad \text{ pour } X \in i\mathfrak{h}_K, Y \in \mathfrak{h}_0.
\end{align*}
Son dual fournit une inclusion canonique $\mathfrak{a}_P^* \hookrightarrow \mathfrak{h}^*$. Notons $\mathfrak{h}^1 := \Ker(h_G)$, c'est un sous-espace $W$-invariant. On fix un produit scalaire $W$-invariant $(,)$ sur $\mathfrak{h}$, d'où une norme $\|\cdot\|$, pour lequel $h_P : \mathfrak{h} \twoheadrightarrow$ est une projection orthogonale.

On définit ainsi la fonction $\|\cdot\|_\infty: \tilde{G}_\infty \to \R_{\geq 0}$ par $\|\tilde{x}\|_\infty = e^{\|X\|}$ si $\tilde{x} = \tilde{k}_1 \exp(X) \tilde{k}_2$ avec $\tilde{k}_1, \tilde{k}_2 \in \tilde{K}_\infty$ et $X \in \mathfrak{h}_0$. On suppose, comme c'est loisible, que $\|\cdot\|_\infty$ est la restriction à $\tilde{G}_\infty$ du hauteur adélique $\|\cdot\|$ pour $\tilde{G}$.

Soit $\pi_\infty$ une représentation admissible irréductible de $\tilde{G}_\infty$, on note $\nu_{\pi_\infty}$ son caractère infinitésimal. D'après l'isomorphisme de Harish-Chandra, cela s'identifie à un élément de $\mathfrak{h}_\C^*/W$.

Notons $\mathcal{E}(\mathfrak{h})$ l'espace des distributions (au sens de Schwartz) à support compact sur $\mathfrak{h}$, qui est une algèbre sous le produit convolution $*$. Notons $\mathcal{E}(\mathfrak{h})^W$ son sous-espace de $W$-invariants. Soit $\gamma \in \mathcal{E}(\mathfrak{h})$, on note $\hat{\gamma} \in C^\infty(i\mathfrak{h}^*)$ sa transformée de Fourier, qui se prolonge en une fonction entière sur $\mathfrak{h}_\C^*$. Idem pour $\mathfrak{h}^1$ au lieu de $\mathfrak{h}$. Notre étude de $J^T_\chi$ s'appuie sur le théorème de multiplicateurs suivant.

\begin{theorem}[{\cite[Theorem III.4.2, Corollary III.4.3]{Ar83}}]\label{prop:mult-reel}
  Soient $f_\infty \in \mathcal{H}(\tilde{G}_\infty)$, $\gamma \in \mathcal{E}(\mathfrak{h})^W$. Alors il existe une unique fonction $f_{\infty,\gamma} \in \mathcal{H}(\tilde{G}_\infty)$ telle que pour tout $\pi_\infty \in \Pi_\mathrm{unit}(\tilde{G}_\infty)$, on a
  $$ \pi_\infty(f_{\infty,\gamma}) = \hat{\gamma}(\nu_{\pi_\infty}) \pi_\infty(f_\infty). $$

  Si $f \in \mathcal{H}_N(\tilde{G}_\infty)$ et $\gamma$ est à support dans $\{H \in \mathfrak{h} : \|H\| \leq N_\gamma \}$ pour des $N,N_\gamma \geq 0$, alors $f_\gamma \in \mathcal{H}_{N+N_\gamma}(\tilde{G}_\infty)$.
\end{theorem}
Ce théorème d'Arthur est valable pour tout groupe de Lie dans la classe de Harish-Chandra, y compris les revêtements. Remarquons que, lorsque $\gamma$ est la mesure de Dirac concentrée en $0$, on a $f_\gamma=f$.

Revenons au cas adélique. Pour $\pi = \bigotimes_v \pi_v$ une représentation admissible irréductible de $\tilde{G}$, on pose $\nu_\pi = \nu_{\pi_\infty}$. Soit $f \in \mathcal{H}(\tilde{G}^1)$, qui est la restriction d'une fonction $\sum_{i=1}^s f_{i,\infty} f_i^\infty \in \mathcal{H}(\tilde{G})$, où les $f_{i,\infty}$ et $f_i^\infty$ sont comme d'habitude des fonctions en les composantes archimédiennes et non archimédiennes, respectivement. Pour $\gamma \in \mathcal{E}(\mathfrak{h}^1)^W$, on pose
$$ f_\gamma := \left. \sum_{i=1}^s f_{i,\infty,\gamma} f_i^\infty \right|_{\tilde{G}^1}. $$

\begin{lemma}\label{prop:mult-adelique}
  La fonction $f_\gamma \in \mathcal{H}(\tilde{G}^1)$ est bien définie. Si $f \in \mathcal{H}_N(\tilde{G}^1)$ et $\gamma$ est à support dans $\{H \in \mathfrak{h}^1 : \|H\| \leq N_\gamma \}$ pour des $N,N_\gamma \geq 0$, alors $f_\gamma \in \mathcal{H}_{N+N_\gamma}(\tilde{G}^1)$.
\end{lemma}
\begin{proof}
  Il suffit de prouver la première assertion. Observons d'abord que $\hat{\gamma}$ est invariant par $i\mathfrak{a}_G^*$. Soit $\pi \in \Pi_\mathrm{unit}(\tilde{G})$, l'inversion de Fourier entraîne que
  \begin{align*}
    (\pi|_{\tilde{G}^1})(f_\gamma) & = \int_{i\mathfrak{a}_G^*} \hat{\gamma}(\nu_{\pi_\lambda}) \pi_\lambda\left(\sum_{i=1}^s f_{i,\infty} f_i^\infty\right) \dd\lambda \\
    & = \hat{\gamma}(\nu_\pi) \int_{i\mathfrak{a}_G^*} \pi_\lambda\left(\sum_{i=1}^s f_{i,\infty} f_i^\infty\right) \dd\lambda \\
    & = \hat{\gamma}(\nu_\pi) (\pi|_{\tilde{G}^1})(f|_{\tilde{G}^1}).
  \end{align*}
  Puisque $\pi$ est arbitraire, on voit que $f_\gamma$ est bien défini.
\end{proof}

Fixons $P=MU$, $\chi \in \mathfrak{X}^{\tilde{G}}$, $\pi \in \Pi_\mathrm{unit}(\tilde{M})$ et $\lambda \in i\mathfrak{a}_M^*$ comme précédemment. C'est bien connu que le caractère infinitésimal de la restriction sur $\tilde{G}_\infty$ de $\mathcal{I}_{\tilde{P}}(\lambda)_{\chi,\pi}$ est égale à la $W$-orbite contenant $\nu_\pi + \lambda$. On déduit du Lemme \ref{prop:mult-adelique} et de \eqref{eqn:Psi} le résultat suivant.

\begin{corollary}
  Soient $f \in \mathcal{H}(\tilde{G}^1)$, $\gamma \in \mathcal{E}(\mathfrak{h}^1)^W$ et $T \in \mathfrak{a}_0$ tels que $d_0(T) > C_0$. Alors
  $$ \Psi^T_\pi(\lambda, f_\gamma) = \hat{\gamma}(\nu_\pi + \lambda) \Psi^T_\pi(\lambda, f). $$
\end{corollary}

\paragraph{Des polynômes}
Pour tout $Y \in \mathfrak{h}^1$, on note par $\delta_Y$ la mesure de Dirac concentrée en $Y$. Soit $H \in \mathfrak{h}^1$. Prenons
$$ \gamma := |W|^{-1} \sum_{w \in W} \delta_{w^{-1}H}  $$
où 

\begin{lemma}
  Soient $N \geq 0$, $f \in \mathcal{H}_N(\tilde{G}^1)$, $T \in \mathfrak{a}_0$ tels que la majoration dans la Proposition \ref{prop:majoration-2} est satisfaite pour la fonction $f_\gamma \in \mathcal{H}_{N+\|H\|}(\tilde{G}^1)$. Alors
  $$
    J^T_\chi(f_\gamma) = |W|^{-1} \sum_{w \in W} \sum_{\substack{P \supset P_0 \\ P=MU}} \sum_{\pi \in \Pi_\mathrm{unit}(\tilde{M}^1)} \psi^T_\pi(w^{-1}H) e^{\angles{\nu_\pi, w^{-1}H}},
  $$
  où, pour $\pi \in \Pi_\mathrm{unit}(\tilde{M})$, on définit
  $$ \psi^T_\pi(H) := \int_{i(\mathfrak{a}^G_M)^*} \Psi^T_\pi(\lambda) e^{\angles{\lambda, H}} \dd\lambda. $$
\end{lemma}
Remarquons que $\psi^T_\pi(H)=0$ (où $T,H$ étant variables) sauf pour un nombre fini de $i\mathfrak{a}_M^*$-orbites de $\pi \in \Pi_\mathrm{unit}(\tilde{M})$, car $f$ est supposée $\tilde{K}$-finie.

\begin{proof}
  Il suffit d'appliquer la Proposition \ref{prop:majoration-2}, le Lemme \ref{prop:mult-adelique} et le fait que $\widehat{\delta_Y} = e^{\angles{\cdot,Y}}$ pour tout $Y$.
\end{proof}

En particulier,
\begin{equation}\label{eqn:p}
  p^T(H) := |W|^{-1} \sum_{w \in W} \sum_{P=MU \supset P_0} \sum_{\pi \in \Pi_\mathrm{unit}(\tilde{M}^1)} \psi^T_\pi(w^{-1}H) e^{\angles{\nu_\pi, w^{-1}H}}
\end{equation}
est un polynôme en $T$ pourvu que $d_0(T) > C(1+\|H\|)$, où $C$ est une constante dépendant de $N$. De plus, on a $J^T_\chi(f) = p^T(0)$ pourvu que $d_0(T) > C$.

D'autre part, d'après la Proposition \ref{prop:majoration-1}, $\psi^T_\pi(H)$ et $p^T(H)$ sont lisses en $H$. La difficulté principale est ce que $\psi^T_\pi(H)$ n'est pas une distribution tempérée en $H$. Cela nécessite les constructions suivantes pour isoler les exposants réels.

Soient $P=MU$ comme ci-dessus. Pour $\pi \in \Pi_\mathrm{unit}(\tilde{M})$, son caractère infinitésimal $\nu_\pi$ s'identifie à un élément $\nu_\pi \in \mathfrak{h}_\C^*/W^M$, où $W^M$ est le groupe de Weyl absolu associé à $M_\infty$. Si l'on ne regarde que $\pi|_{\tilde{M}^1}$, alors $\nu_\pi$ est déterminé à $i\mathfrak{a}_M^*$ près (rappelons que $\mathfrak{a}_M^* \hookrightarrow \mathfrak{h}^*$ via le dual de $h_P$). Dorénavant, nous en fixons un représentant, noté abusivement $\nu_\pi$, tel que
$$ \nu_\pi = X_\pi + i Y_\pi, \quad X_\pi, Y_\pi \in \mathfrak{h}^*$$
avec $\|Y_\pi\|$ minimal.

On dit que deux paires $(w_1, \pi_1)$, $(w_2, \pi_2)$ avec $w_i \in W$, $\pi_i \in \Pi_\mathrm{unit}(\tilde{M}^1)$ ($i=1,2$) sont équivalentes si $w_1 X_{\pi_1} = w_2 X_{\pi_2}$. L'ensemble de telles classes d'équivalence est noté $\mathcal{E}$. Pour tout $\Gamma = [w,\pi] \in \mathcal{E}$ (ne pas le confondre avec le même symbole désignant les $\tilde{K}$-types dans \S\ref{sec:prod-Eis}), l'élément $X_\Gamma := wX_\pi$ est bien défini. On peut regrouper les termes dans \eqref{eqn:p} et écrire
$$ p^T(H) = \sum_{\Gamma \in \mathcal{E}} \psi^T_\Gamma(H) e^{\angles{X_\Gamma, H}}, \quad d_0(T) > C(1+\|H\|) $$
où
\begin{equation}\label{eqn:psi-Gamma}
  \psi^T_\Gamma(H) := |W|^{-1} \sum_{(w, \pi) \in \Gamma} \psi^T_\pi(w^{-1}H) e^{\angles{iY_\pi, w^{-1}H}}.
\end{equation}

\begin{lemma}\label{prop:psi-Gamma-majoration}
  Soit $D$ un opérateur différentiel à coefficients constants sur $\mathfrak{h}^1$. Il existe une constante $c_D > 0$ telle que pour tout $\Gamma \in \mathcal{E}$, $H \in \mathfrak{h}^1$ et $T \in \mathfrak{a}_0$ tel que $d_0(T) > C_0$, on a
  $$ |D\psi_\Gamma^T(H)| \leq c_D (1+\|T\|)^{d_0}. $$

  Il en résulte que $p^T(H)$ est de degré $\leq d_0$ en $T$.
\end{lemma}
\begin{proof}
  C'est une conséquence immédiate de la Proposition \ref{prop:majoration-1}.
\end{proof}

\begin{proposition}\label{prop:p-Gamma}
  Il existe une unique famille de fonctions $(p^T_\Gamma(H))_{\Gamma \in \mathcal{E}}$ qui sont polynomiales de degrés $\leq d_0$ en $T$ et lisses en $H$ telle que
  \begin{enumerate}
    \item pour tout $T,H$, on a
    $$ p^T(H) = \sum_{\Gamma \in \mathcal{E}} p^T_\Gamma(H) e^{\angles{X_\Gamma, H}}; $$
    \item il existe des constantes $C, \epsilon >0$ telles que pour tout opérateur différentiel $D$ à coefficients constants sur $\mathfrak{h}^1$, il existe une constante $c_D > 0$ vérifiant
    \begin{equation}\label{eqn:p-Gamma-1}
      |D(\psi_\Gamma^T(H) - p^T_\Gamma(H))| \leq c_D e^{-\epsilon d_0(T)} (1+\|T\|)^{d_0}
    \end{equation}
    pour tous $\Gamma, H, T$ avec $d_0(T) > C(1+\|H\|)$, et
    \begin{equation}\label{eqn:p-Gamma-2}
      |D p^T_\Gamma(H)| \leq c_D (1+\|H\|)^{d_0} (1+\|T\|)^{d_0}
    \end{equation}
    pour tous $H, T$.
  \end{enumerate}
  En particulier, pour tout $T \in \mathfrak{a}_0$, on a
  $$ J^T_\chi(f) = \sum_{\Gamma \in \mathcal{E}} p^T_\Gamma(0), $$
  et $p^T_\Gamma(H)$ est une distribution tempérée sur $\mathfrak{h}^1$ pour tout $T$.
\end{proposition}
\begin{proof}
  Il découle du résultat général \cite[Proposition 5.1]{Ar82-Eis1}.
\end{proof}

Soit $V$ est un $\R$-espace vectoriel, on note $\mathcal{S}(V)$ l'espace des fonctions de Schwartz sur $V$. Soient $\Gamma \in \mathcal{E}$, $\beta \in \mathcal{S}(\mathfrak{h}^1)$. Vu la Proposition \ref{prop:p-Gamma}, on peut définir
$$ p^T_\Gamma(\beta) := \int_{\mathfrak{h}^1} p^T_\Gamma(H)\beta(H) \dd H .$$

Pour tout $\epsilon > 0$, on définit la fonction $\beta_\epsilon: H \mapsto \epsilon^{-\dim\mathfrak{h}^1} \beta(\epsilon^{-1}H) $. Supposons de plus que $\int_{\mathfrak{h}^1} \beta(H)\dd H = 1$. C'est un fait standard que
$$ \lim_{\epsilon \to 0} p^T_\Gamma(\beta_\epsilon) = p^T_\Gamma(0). $$

\begin{definition}\label{def:limite-T}
  Soit $P=MU \in \mathfrak{F}(M_0)$. Si $g: \mathfrak{a}_M \to \C$ est une fonction et $L \in \C$, l'expression
  $$ \lim_{T \xrightarrow{P} \infty} g(T) = L  $$
  signifie que, pour tous $\epsilon, \eta > 0$, il existe $R > 0$ tel que $|g(T)-L| < \epsilon$ pourvu que
  \begin{itemize}
    \item $\angles{\alpha, T} > R$ pour tout $\alpha \in \Sigma_P$;
    \item $\angles{\alpha-\eta\beta, T} > 0$ pour tous $\alpha, \beta \in \Sigma_P$.
  \end{itemize}
\end{definition}
Cela est une version précise de la notion ``$T \to \infty$ fortement dans $\mathfrak{a}_P^+$'' dans \cite[p.1272]{Ar82-Eis1}.

\begin{lemma}
  Pour tout $\beta \in \mathcal{S}(\mathfrak{h}^1)$, on a
  $$ \lim_{T \xrightarrow{P_0} \infty} \left( \int_{\mathfrak{h}^1} \psi^T_\Gamma(H)\beta(H)\dd H - p^T_\Gamma(\beta) \right) = 0. $$
\end{lemma}
\begin{proof}
  C'est un exercice en analyse. On a
  \begin{align*}
    \left| \int_{\mathfrak{h}^1} \psi^T_\Gamma(H)\beta(H)\dd H - p^T_\Gamma(\beta) \right| & \leq \int_{\mathfrak{h}_1} |\psi_\Gamma^T(H) - p_\Gamma^T(H)|\cdot |\beta(H)| \dd H \\
    & = \int_{\{H : d_0(T) \leq C(1+\|H\|)\}} (\cdots) + \int_{\{H : d_0(T) > C(1+\|H\|)\}} (\cdots).
  \end{align*}

  Considérons la première intégrale. Puisqu'on considère la limite $T \xrightarrow{P_0} \infty$, d'après \eqref{eqn:p-Gamma-2} et le Lemme \ref{prop:psi-Gamma-majoration}, $|\psi^T_\Gamma(H) - p^T_\Gamma(H)|$ est bornée par $C'(1+\|T\|)^{d_0}(1+\|H\|)^{d_0}$ pour une constante $C'$. Vu la Définition \ref{def:limite-T}, on peut borner $\|T\|$ par $d_0(T)$ multiplié par une certaine constante et il existe une constante $C''$ telle que
  $$ \|T\| \leq C'' (1+\|H\|) $$
  pour tout $H$ dans le domaine d'intégration. Soit $n \in \Z$, $n > 0$. La première intégrale est donc bornée par
  $$ C_n \|T\|^{-n} \int_{\mathfrak{h}^1} |\beta(H)| (1+\|H\|)^{2 d_0 + n} \dd H $$
  où $C_n$ est une autre constante. Comme $\beta \in \mathcal{S}(\mathfrak{h}^1)$, cette expression tend vers $0$.

  Pour la deuxième intégrale, on utilise \eqref{eqn:p-Gamma-1}. Elle est bornée par
  $$ c_1 e^{-\epsilon d_0(T)} (1+\|T\|)^{d_0} \int_{\mathfrak{h}^1} |\beta(H)| \dd H. $$
  Cela tend vers $0$ lorsque $T \xrightarrow{P_0} \infty$. D'où le lemme.
\end{proof}

\begin{theorem}[{\cite[Theorem 6.3]{Ar82-Eis1}}]\label{prop:P^T(B)-J^T}
  Soit $B \in \mathcal{S}(i\mathfrak{h}^*/i\mathfrak{a}_G^*)^W$. Pour $\pi \in \Pi_\mathrm{unit}(\tilde{M}^1)$ on pose
  $$ B_\pi(\lambda) := B(iY_\pi + \lambda), \quad \lambda \in i(\mathfrak{a}^G_M)^* $$
  qui est bien défini, et pour $\epsilon > 0$ on pose
  $$ B^\epsilon(\lambda) := B(\epsilon\lambda). $$

  Alors Il existe un unique polynôme $P^T(B)$ en $T$ tel que
  $$ \lim_{T \xrightarrow{P_0} \infty} \left( \sum_{\substack{P=MU \\ P \supset P_0}} \sum_{\pi \in \Pi_\mathrm{unit}(\tilde{M}^1)}\; \int_{i(\mathfrak{a}^G_M)^*} \Psi^T_\pi(\lambda)B_\pi(\lambda)\dd\lambda - P^T(B) \right) = 0 . $$

  De plus, si $B(0)=1$ alors
  $$ J^T_\chi(f) = \lim_{\epsilon \to 0} P^T(B^\epsilon). $$
\end{theorem}
\begin{proof}
  Montrons d'abord l'unicité. Soient $P^T(B)$, $Q^T(B)$ deux polynômes en $T$ satisfaisant à la propriété du Théorème, alors
  $$ \lim_{T \xrightarrow{P_0} \infty} (P^T(B) - Q^T(B))=0, $$
  or cela entraîne que $P^T(B)=Q^T(B)$.

  Pour l'existence, on prend $\beta \in S(\mathfrak{h}^1)$ tel que $B = \hat{\beta}|_{i\mathfrak{h}^*/i\mathfrak{a}_G^*}$. Alors $B^\epsilon = \widehat{\beta_\epsilon}|_{i\mathfrak{h}^*/i\mathfrak{a}_G^*}$. Posons
  $$ P^T(B) := \sum_{\Gamma \in \mathcal{E}} p^T_\Gamma(\beta). $$
  L'assertion résulte de ce qui précèdent.
\end{proof}

Maintenant on est en mesure d'employer les résultats du \S\ref{sec:prod-Eis}. Soient $T \in \mathfrak{a}_0$, $\lambda,\lambda' \in i\mathfrak{a}_M^*$. Définissons l'opérateur suivant de $\mathcal{A}^2(\tilde{P})_{\chi,\pi}$ sur lui-même:
\begin{equation}\label{eqn:omega-def}
  \omega_{\chi,\pi}^T(\tilde{P},\lambda', \lambda) := \sum_{\substack{P_1=M_1U_1 \\ P_1 \supset P_0}} \sum_{\substack{t \in W(\mathfrak{a}_M, \mathfrak{a}_{M_1}) \\ t' \in W(\mathfrak{a}_M, \mathfrak{a}_{M_1})}} M(t,\lambda)^{-1} M(t',\lambda') \frac{e^{\angles{t'\lambda'-t\lambda, T}}}{\theta_{P_1}(t'\lambda'-t\lambda)} .
\end{equation}

Vu les résultats dans \S\ref{sec:prod-Eis}, cet opérateur n'a pas de pôles pour $\lambda,\lambda' \in i\mathfrak{a}_M^*$. Posons
$$ \omega_{\chi,\pi}^T(\tilde{P},\lambda) := \omega_{\chi,\pi}^T(\tilde{P},\lambda,\lambda), \quad \lambda \in i\mathfrak{a}_M^* . $$

\begin{theorem}[{\cite[Theorem 7.1]{Ar82-Eis1}}]\label{prop:P^T(B)-asymp}
  Soit $B \in C_c^\infty(i\mathfrak{h}^*/i\mathfrak{a}_G^*)^W$. Alors $P^T(B)$ est l'unique polynôme en $T$ tel que
  $$ \lim_{T \xrightarrow{P_0} \infty} \left( \sum_{P \supset P_0} \sum_{\pi \in \Pi_\mathrm{unit}(\tilde{M}^1)}\; |\mathcal{P}(M)|^{-1} \int_{i(\mathfrak{a}^G_M)^*} \Tr(\omega^T_{\chi,\pi}(\tilde{P},\lambda)\mathcal{I}_{\tilde{P}}(\lambda, f)_{\chi,\pi}) B_\pi(\lambda) \dd\lambda - P^T(B) \right) = 0 . $$
\end{theorem}
Observons que la somme sur $\pi$ est finie et les traces se calculent dans des espaces de dimension finie, grâce à la $\tilde{K}$-finitude de $f$.
\begin{proof}
  Rappelons d'abord que
  $$ \sum_{P,\pi} \int_{i(\mathfrak{a}^G_M)^*} \Psi^T_\pi(\lambda)B_\pi(\lambda) \dd\lambda = \sum_{P,\pi} |\mathcal{P}(M)|^{-1} \int_{i(\mathfrak{a}^G_M)^*} \Tr(\Omega^T_{\chi,\pi}(\tilde{P},\lambda)\mathcal{I}_{\tilde{P}}(\lambda,f)_{\chi,\pi}) B_\pi(\lambda) \dd\lambda .$$

  Vu la définition de $\Omega^T_{\chi,\pi}(\tilde{P},\lambda)$ et le Corollaire \ref{prop:prod-Eis-3}, il existe une fonction localement bornée $\rho_\pi: i\mathfrak{a}_M^* \to \R_{\geq 0}$ telle que la différence
  $$ \int_{i(\mathfrak{a}^G_M)^*} \Tr(\Omega^T_{\chi,\pi}(\tilde{P},\lambda)\mathcal{I}_{\tilde{P}}(\lambda,f)_{\chi,\pi}) B_\pi(\lambda) \dd\lambda - \int_{i(\mathfrak{a}^G_M)^*} \Tr(\omega^T_{\chi,\pi}(\tilde{P},\lambda)\mathcal{I}_{\tilde{P}}(\lambda,f)_{\chi,\pi}) B_\pi(\lambda) \dd\lambda $$
  est bornée par
  $$ e^{-\epsilon\|T\|} \int_{i(\mathfrak{a}^G_M)^*} \rho_\pi(\lambda) B_\pi(\lambda) \dd\lambda. $$

  Puisque la somme sur $P,\pi$ est finie et $B$ est à support compact, l'assertion en résulte.
\end{proof}

\section{Formule explicite}\label{sec:formule-explicite}
Fixons toujours $f \in \mathcal{H}(\tilde{G}^1)$, $\chi \in \mathfrak{X}^{\tilde{G}}$ et $B \in C_c^\infty(i\mathfrak{h}^*/i\mathfrak{a}_G^*)^W$ tel que $B(0)=1$. Dans cette section, nous nous proposons d'obtenir une formule explicite pour $P^T(B)$, et puis évaluer $\lim_{\epsilon \to 0} P^T(B^\epsilon)$, qui est égal à $J^T_\chi(f)$.

\paragraph{Les opérateurs $M_{P_1|P}(w,\lambda)$}
Soient $P=MU$, $P_1 = M_1 U_1 \in \mathcal{F}(M_0)$, $w \in W(\mathfrak{a}_M, \mathfrak{a}_{M_1})$. Nous allons ḍéfinir une famille d'opérateurs d'entrelacement $M_{P_1|P}(w,\lambda): \mathcal{I}_{\tilde{P}}(\lambda) \to \mathcal{I}_{\tilde{P}_1}(w\lambda)$ qui est méromorphe en $\lambda \in \mathfrak{a}_{M,\C}^*$.

Considérons d'abord le cas $w=1$, alors $M=M_1$. On pose
$$ (M_{P_1|P}(1,\lambda)\phi)(\tilde{x}) := \int_{(U_1 \cap U)(\A) \backslash U_1(\A)} \phi(u\tilde{x}) e^{\angles{\lambda, H_P(ux)-H_{P_1}(x)}} \dd u $$
pour tous $\phi \in \mathcal{A}^2(\tilde{P})$, $\tilde{x} \in \tilde{G}$. Cette intégrale est absolument convergente si $\angles{\Re \lambda, \alpha^\vee} \gg 0$ pour tout $\alpha \in \Delta_P$. Écrivons aussi $M_{P_1|P}(\lambda) := M_{P_1|P}(1,\lambda)$.

Considérons le cas général. Fixons un représentant $\hat{w} \in G(F)$ de $w$. Définissons $A(w,\lambda): \mathcal{A}^2(w^{-1}\tilde{P}_1 w) \rightiso \mathcal{A}^2(\tilde{P}_1)$ par
$$ (A(w,\lambda)\phi)(\tilde{x}) = \phi(\hat{w}^{-1} \tilde{x}) e^{\angles{\lambda, (w^{-1}-1)T_0}}. $$

Posons
$$ M_{P_1 | P}(w,\lambda) := A(w,\lambda) M_{w^{-1} P_1 w | P}(1, \lambda). $$
C'est défini par une intégrale absolument convergente si $\Re\lambda$ appartient à un cône ouvert dans $\mathfrak{a}_M^*$. Nous laissons le soin au lecteur de vérifier que ladite définition coïncide avec celle dans \cite{Ar82-Eis2}, à savoir
$$ (M_{P_1|P}(w,\lambda)\phi)(\tilde{x}) = \int_{(U_1 \cap wUw^{-1})(\A) \backslash U_1(\A)} \phi(\hat{w}^{-1} u\tilde{x}) e^{\angles{\lambda, H_P(\hat{w}^{-1}ux)} - \angles{w\lambda, H_{P_1}(x)}} \dd u . $$

On vérifie que cela ne dépend que de la classe $w W^M_0$. Lorsque $P,P_1$ sont standards, ils sont déterminés par $M$ et $w$, et on voit que
$$ M(w,\lambda) = M_{P_1|P}(w,\lambda). $$

Des arguments standards entraînent que les opérateurs $M_{P_1|P}(w,\lambda)$ satisfont aux mêmes propriétés analytiques que $M(w,\lambda)$. En particulier, ils se prolongent en des opérateurs méromorphes en $\lambda$ et réguliers pour $\lambda \in i\mathfrak{a}_M^*$, et on a les équations fonctionnelles suivantes.
\begin{enumerate}
  \item Soient $P, P_1, P_2 \in \mathcal{F}(M_0)$, $w \in W(\mathfrak{a}_M, \mathfrak{a}_{M_1})$ et $w_1 \in W(\mathfrak{a}_{M_1}, \mathfrak{a}_{M_2})$, alors
    $$ M_{P_2|P}(w_1 w, \lambda) = M_{P_2|P_1}(w_1, w\lambda) M_{P_1|P}(w,\lambda); $$
  \item Soient $P,P_1,w$ comme précédemment et $\phi \in \mathcal{A}^2(\tilde{P})$, on a
    $$ E(\phi,\lambda) = E(M_{P_1|P}(w,\lambda)\phi, w\lambda). $$
\end{enumerate}

Maintenant, soit $Q \in \mathcal{P}(M)$. Il existe un unique parabolique standard $P_1$ conjugué à $Q$, disons $Q = t^{-1} P_1 t$ où $t$ est unique en tant qu'un élément de $W^G_0/W^M_0$. Pour $T \in \mathfrak{a}_0$, on peut bien définir
$$ Y_Q(T) := \text{ la projection de } t^{-1}(T-T_0) + T_0 \text{ sur } \mathfrak{a}_M . $$

\begin{proposition}\label{prop:omega-prod-s}
  Soient $P=MU \in \mathcal{F}(M_0)$, $\pi \in \Pi_\mathrm{unit}(\tilde{M}^1)$ et $\lambda \in i\mathfrak{a}_M^*$, alors $\omega^T_{\chi,\pi}(\tilde{P},\lambda)$ est égal à la valeur en $\lambda'=\lambda$ de
  $$ \sum_{s \in W(M)} \sum_{Q \in \mathcal{P}(M)} M_{Q|P}(\lambda)^{-1} M_{Q|P}(s,\lambda') \frac{e^{\angles{s\lambda'-\lambda, Y_Q(T)}}}{\theta_Q(s\lambda'-\lambda)}. $$
\end{proposition}
\begin{proof}
  Dans la définition \eqref{eqn:omega-def} de $\omega^T_{\chi,\pi}(\tilde{P},\lambda)$, on écrit les éléments $t' \in W(\mathfrak{a}_M, \mathfrak{a}_{M_1})$ comme $t'=ts$ où $s \in W(\mathfrak{a}_M,\mathfrak{a}_M) = W(M)$. Supposons désormais que $\lambda,\lambda' \in i\mathfrak{a}_M^*$ et posons $Q := t^{-1} P_1 t$. Pour conclure, il reste à vérifier que
  \begin{align*}
    M_{P_1|P}(t,\lambda)^{-1} M_{P_1|P}(ts,\lambda') &= M_{Q|P}(\lambda)^{-1} M_{Q|P}(s,\lambda') e^{\angles{s\lambda'-\lambda, T_0 - t^{-1}T_0}}, \\
    \theta_{P_1}(ts\lambda' -t\lambda) & = \theta_Q(s\lambda'-\lambda), \\
    \angles{ts\lambda' - t\lambda, T} + \angles{s\lambda'-\lambda, T_0 - t^{-1} T_0} & = \angles{s\lambda'-\lambda, Y_Q(T)}.
  \end{align*}
\end{proof}

\begin{corollary}
  Conservons les notations précédentes, $\Tr(\omega^T_{\chi,\pi}(\tilde{P},\lambda) \mathcal{I}_{\tilde{P}}(\lambda,f)_{\chi,\pi})$ est égal à la valeur en $\lambda'=\lambda$ de
  \begin{equation}\label{eqn:cd-prim}
    \sum_{s \in W(M)} \sum_{Q \in \mathcal{P}(M)} \Tr(M_{Q|P}(\lambda)^{-1} M_{Q|P}(s,\lambda')\mathcal{I}_{\tilde{P}}(\lambda,f)_{\chi,\pi}) \frac{e^{\angles{s\lambda'-\lambda, Y_Q(T)}}}{\theta_Q(s\lambda'-\lambda)}.
  \end{equation}
\end{corollary}

\paragraph{Des $(G,M)$-familles}
Fixons $M \in \mathcal{L}(M_0)$, qui n'est pas forcément un Lévi standard pour l'instant, et $\pi \in \Pi_\mathrm{unit}(\tilde{M}^1)$. Rappelons que nous avons défini le sous-ensemble $W^L(M)_\text{reg} \subset W^L(M)$ dans \S\ref{sec:prem}, où $L \in \mathcal{L}(M)$. On a une décomposition
$$ W(M) = \bigsqcup_{L \in \mathcal{L}(M)} W^L(M)_\text{reg}. $$

Soit $s \in W(M)$. Prenons l'unique $L \in \mathcal{L}(M)$ tel que $s \in W^L(M)_\text{reg}$. On a un isomorphisme
\begin{align*}
  \mathfrak{a}_M^* \oplus \mathfrak{a}_L^* & \longrightarrow \mathfrak{a}_M^* \oplus \mathfrak{a}_L^*, \\
  (\lambda, \zeta) & \longmapsto (\Lambda, \lambda_L),
\end{align*}
où $\Lambda := (s-1)\lambda+\zeta$ et $\lambda_L$ est la projection de $\lambda$ sur $\mathfrak{a}_L^*$. Fixons $P \in \mathcal{P}(M)$. Soient $(\lambda, \zeta)$ comme ci-dessus et $T \in \mathfrak{a}_0$, définissons deux familles de fonctions
\begin{align*}
  c_Q(T,\Lambda) & := e^{\angles{\Lambda, Y_Q(T)}}, \\
  d_Q(\lambda_L, \Lambda) & := \Tr(M_{Q|P}(\lambda)^{-1} M_{Q|P}(s, \lambda+\zeta) \mathcal{I}_{\tilde{P}}(\lambda,f)_{\chi,\pi}), \quad Q \in \mathcal{P}(M).
\end{align*}

\begin{lemma}
  Les fonctions $c_Q(T, \Lambda)$, $d_Q(\lambda_L, \Lambda)$ forment des $(G,M)$-familles des fonctions en $\Lambda \in i\mathfrak{a}_M^*$.
\end{lemma}
\begin{proof}
  C'est pareil que \cite[\S 2]{Ar82-Eis1}. La famille $\{c_Q(T,\Lambda)\}_{Q \in \mathcal{P}(M)}$ est exactement celle utilisée par Arthur et le revêtement n'y intervient pas. L'assertion concernant $d_Q(\lambda_L, \Lambda)$ est prouvée en utilisant des propriétés de base des opérateurs d'entrelacement $M_{Q|P}(\cdots)$.
\end{proof}

Maintenant, soient $\lambda, \lambda' \in i\mathfrak{a}_M^*$ comme dans \eqref{eqn:cd-prim}. Puisque nous ne regardons que la limite $\lambda' \to \lambda$, c'est loisible d'écrire $\lambda' = \lambda+\zeta$ avec $\zeta \in i\mathfrak{a}_L^*$. Via la bijection $(\lambda, \zeta) \mapsto (\Lambda, \lambda_L)$ ci-dessus, on a
$$ \Lambda = (s-1)\lambda+\zeta = s\lambda' - \lambda . $$

Supposons maintenant que $P$ est un parabolique standard. Le terme associée à $s$ dans \eqref{eqn:cd-prim} devient
$$ \sum_{Q \in \mathcal{P}(M)} c_Q(T,\Lambda) d_Q(\lambda_L, \Lambda) \theta_Q(\Lambda)^{-1}. $$

D'après la formule de descente \cite[Lemme 4.2.4]{Li10a}, sa valeur en $\lambda'=\lambda$ est égale à
$$ \sum_{S \in \mathcal{F}(M)} c^S_M(T, (s-1)\lambda) d'_S(\lambda_L, (s-1)\lambda). $$

On en déduit que
\begin{multline*}
  \int_{i(\mathfrak{a}^G_M)^*} \Tr(\omega^T_{\chi,\pi}(\tilde{P},\lambda)\mathcal{I}_{\tilde{P}}(\lambda,f)_{\chi,\pi}) B_\pi(\lambda) \dd\lambda \\
  = \sum_{L \in \mathcal{L}(M)} \sum_{s \in W^L(M)_\text{reg}} \; \int_{i(\mathfrak{a}^G_M)^*} \sum_{S \in \mathcal{F}(M)} c^S_M(T, (s-1)\lambda) d'_S(\lambda_L, (s-1)\lambda) B_\pi(\lambda) \dd\lambda \\
  = \sum_{L,s} |\det(s-1|\mathfrak{a}^L_M)|^{-1} \sum_{S \in \mathcal{F}(M)} \;\int_{i(\mathfrak{a}^L_M)^*} \int_{i(\mathfrak{a}^G_L)^*} c^S_M(T,\mu) d'_S(\lambda, \mu) B_\pi((s-1)^{-1}\mu + \lambda) \dd\lambda \dd\mu
\end{multline*}
à l'aide de l'isomorphisme $(s-1): i(\mathfrak{a}^L_M)^* \to i(\mathfrak{a}^L_M)^*$. N'oublions pas que, vu le Théorème \ref{prop:P^T(B)-asymp}, il suffit de considérer le comportement de cette expression lorsque $T \xrightarrow{P_0} \infty$, ce qui est dicté par la $(G,M)$-famille $(c_Q(T,\Lambda))_{Q \in \mathcal{P}(M)}$ qui n'a rien à faire avec les revêtements. Les arguments dans \cite[pp.1305-1308]{Ar82-Eis1} sont toujours applicables car la seule propriété de $d'_S(\lambda,\mu)$ qui importe est sa lissité. Ils entraînent que la différence entre
$$ \int_{i(\mathfrak{a}^G_M)^*} \Tr(\omega^T_{\chi,\pi}(\tilde{P},\lambda)\mathcal{I}_{\tilde{P}}(\lambda,f)_{\chi,\pi}) B_\pi(\lambda) \dd\lambda $$
et
$$ \sum_{L,s} |\det(s-1|\mathfrak{a}^L_M)|^{-1} \int_{i(\mathfrak{a}^G_L)^*} \left( \sum_{S \in \mathcal{F}(L)} c^S_L(T) d'_S(\lambda) \right) B_\pi(\lambda) \dd\lambda $$
tend vers $0$ lorsque $T \xrightarrow{P_0} \infty$, et que $c^S_L(T)$ est un polynôme en $T$ pour tout $S \in \mathcal{F}(L)$. En appliquant encore une fois la formule de descente, on obtient
$$ \sum_{S \in \mathcal{F}(L)} c^S_L(T,\Lambda) d'_S(\lambda,\Lambda) = \sum_{Q_1 \in \mathcal{P}(L)} c_{Q_1}(T,\Lambda) d_{Q_1}(\lambda, \Lambda) \theta_{Q_1}(\Lambda)^{-1}, \quad \lambda,\Lambda \in i\mathfrak{a}_L^* . $$

Définissons des opérateurs
\begin{align*}
  \mathcal{M}_Q(\tilde{P}, \lambda, \Lambda) & := M_{Q|P}(\lambda)^{-1} M_{Q|P}(\lambda+\Lambda), \\
  \mathcal{M}^T_Q(\tilde{P}, \lambda, \Lambda) & := c_Q(T,\Lambda) \mathcal{M}_Q(P, \lambda, \Lambda), \qquad Q \in \mathcal{P}(M)
\end{align*}
où $\lambda, \Lambda \in i\mathfrak{a}_M^*$ sont supposés en position générale. On vérifie à l'aide des propriétés des opérateurs d'entrelacement qu'ils forment des $(G,M)$-familles en $\Lambda$.

Soit $Q_1 \in \mathcal{P}(L)$, on choisit $Q \in \mathcal{P}(M)$ avec $Q \subset Q_1$. Si $\lambda, \Lambda \in i\mathfrak{a}_L^*$, on vérifie que
$$ M_{Q|P}(s, \lambda+\Lambda) = M_{Q|P}(\lambda+\Lambda) M_{P|P}(s, \lambda+\Lambda) = M_{Q|P}(\lambda+\Lambda) M_{P|P}(s,0), $$
d'où
\begin{align*}
  d_{Q_1}(\lambda,\Lambda) &= \Tr(M_{Q|P}(\lambda)^{-1} M_{Q|P}(s,\lambda+\Lambda) \mathcal{I}_{\tilde{P}}(\lambda,f)_{\chi,\pi}) \\
  & = \Tr(M_{Q|P}(\lambda)^{-1} M_{Q|P}(\lambda+\Lambda) M_{P|P}(s,0) \mathcal{I}_{\tilde{P}}(\lambda,f)_{\chi,\pi}).
\end{align*}

Donc $c_{Q_1}(T,\Lambda)d_{Q_1}(\lambda, \Lambda)$ est égal à
$$ \Tr(\mathcal{M}_{Q_1}^T(\tilde{P}, \lambda, \Lambda) M_{P|P}(s,0) \mathcal{I}_{\tilde{P}}(\lambda,f)_{\chi,\pi}). $$

Vu le Théorème \ref{prop:P^T(B)-asymp}, le bilan de ces raisonnements est le suivant.

\begin{theorem}[Cf. {\cite[Theorem 4.1]{Ar82-Eis2}}]\label{prop:formule-explicite}
  $P^T(B)$ est égal à
  \begin{multline}\label{eqn:P^T(B)-fin}
    \sum_{\substack{P \in \mathcal{F}(M_0) \\ P=MU \supset P_0}} \sum_{\pi \in \Pi_\mathrm{unit}(\tilde{M}^1)} \sum_{L \in \mathcal{L}(M)} \sum_{s \in W^L(M)_\text{reg}} |\mathcal{P}(M)|^{-1} |\det(s-1|\mathfrak{a}^L_M)|^{-1} \cdot \\
    \cdot \int_{i(\mathfrak{a}^G_L)^*} \Tr(\mathcal{M}^T_L(\tilde{P},\lambda) M_{P|P}(s,0) \mathcal{I}_{\tilde{P}}(\lambda,f)_{\chi,\pi}) B_\pi(\lambda) \dd\lambda .
  \end{multline}
\end{theorem}

\section{Convergence absolue}\label{sec:conv}
L'étape suivante est d'évaluer $\lim_{\epsilon \to 0} P^T(B^\epsilon)$ à l'aide de \eqref{eqn:P^T(B)-fin}. Pour ce faire, il faudra des majorations qui permettront d'appliquer la convergence dominée à \eqref{eqn:P^T(B)-fin}.

\paragraph{Normalisation des opérateurs d'entrelacement}
Soient $M \in \mathcal{L}(M_0)$, $P,Q \in \mathcal{P}(M)$ et $\pi = \bigotimes_v \pi_v \in \Pi_\mathrm{unit}(\tilde{M}^1)$. Ici on regarde $\pi$ comme une représentation de $\tilde{M}$ sur laquelle $A_{M,\infty}$ opère trivialement. Supposons que $\pi$ intervient dans le spectre discret $L^2_\text{disc}(M(F) \backslash \tilde{M}^1)$, c'est-à-dire la somme directe complétée des sous-représentations irréductibles de $L^2(M(F) \backslash \tilde{M}^1)$. En chaque place $v$ de $F$, nous avons défini dans \cite[\S 3]{Li11a} les objets locaux suivants.

\begin{itemize}
  \item Les opérateurs d'entrelacement $J_{\tilde{Q}|\tilde{P}}(\pi_{v,\lambda}): \mathcal{I}_{\tilde{P}}(\pi_{v,\lambda}) \to \mathcal{I}_{\tilde{Q}}(\pi_{v,\lambda})$, méromorphes en $\lambda \in \mathfrak{a}_{M,\C}^*$.
  \item Les facteurs normalisant $r_{\tilde{Q}|\tilde{P}}(\pi_\lambda)$, qui sont des fonctions méromorphes en $\lambda \in \mathfrak{a}_{M,\C}^*$. Lorsque $v \notin V_\text{ram}$ et $\pi_v$ est non ramifié, nous prenons les facteurs normalisants non ramifiés (voir \cite[\S 3.4]{Li11a}).
  \item Les opérateurs d'entrelacement normalisés $R_{\tilde{Q}|\tilde{P}}(\pi_{v,\lambda}) := r_{\tilde{Q}|\tilde{P}}(\pi_\lambda)^{-1} J_{\tilde{Q}|\tilde{P}}(\pi_{v,\lambda})$. Ils sont méromorphes en $\lambda$ et sont des opérateurs unitaires pour $\lambda \in i\mathfrak{a}_M^*$.
\end{itemize}

En fait, dans \cite{Li11a} on ne considère que les représentations spécifiques, c'est-à-dire le cas où $\pi(\noyau) = \noyau\cdot\identity$ pour tout $\noyau \in \bmu_m$. Or on peut toujours se ramener au cas spécifique en passant à un revêtement plus petit. De plus, on a la décomposition
\begin{align}\label{eqn:r-prod}
  r_{\tilde{Q}|\tilde{P}}(\pi_{v,\lambda}) & = \prod_{\alpha \in \Sigma_Q^{\text{red}} \cap \Sigma_{\bar{P}}^{\text{red}}} r_\alpha(\pi_{v,\lambda}), \\
  r_\alpha(\pi_{v,\lambda}) & := r^{\tilde{M}_\alpha}_{\tilde{Q} \cap \tilde{M}_\alpha | \tilde{P} \cap \tilde{M}_\alpha}(\pi_{v,\lambda}),
\end{align}
où $M_\alpha$ est le sous-groupe de Lévi tel que $\Sigma^{M_\alpha, \text{red}}_{Q \cap M_\alpha} = \{\alpha\}$.

Soit $\phi = \bigotimes \phi_v$ un élément dans l'espace sous-jacent de $\mathcal{I}_{\tilde{P}}(\pi_\lambda) := {\bigotimes_v}' \mathcal{I}_{\tilde{P}}(\pi_{v,\lambda})$. Pour presque toute place $v \notin V_\text{ram}$ telle que $\pi_v$ est non ramifiée, rappelons que la commutativité de l'algèbre de Hecke sphérique garantit qu'il n'y a qu'une seule droite de vecteurs sphériques de $\pi_v$, et $\phi_v$ est le vecteur sphérique que nous fixons pour définir le produit tensoriel restreint. Or, d'après la construction des opérateurs $R_{\tilde{Q}|\tilde{P}}(\pi_{v,\lambda})$, on voit que $R_{\tilde{Q}|\tilde{P}}(\pi_{v,\lambda})\phi_v = \phi_v$ pour tout $\lambda \in \mathfrak{a}_{M,\C}^*$ et presque toute place $v$. Donc l'opérateur d'entrelacement normalisé global
$$ R_{Q|P}(\pi_\lambda) := \prod_v R_{\tilde{Q}|\tilde{P}}(\pi_{v,\lambda}) $$
est bien défini: son action sur chaque vecteur lisse est effectivement donnée par un produit fini.

D'autre part, $\mathcal{A}^2(\tilde{P})_\pi$ s'identifie à un sous-espace dense de $\Hom(\pi, L^2_\text{disc}(M(F) \backslash \tilde{M}^1)) \otimes \mathcal{I}_{\tilde{P}}(\lambda)$. Idem pour $Q$ au lieu de $P$. En comparant les formules définissant $M_{Q|P}(\lambda)$ et celles pour les $J_{\tilde{Q}|\tilde{P}}(\pi_{v,\lambda})$, on voit que
$$ \left. \identity \otimes \prod_v J_{\tilde{Q}|\tilde{P}}(\pi_{v,\lambda}) \right|_{\mathcal{A}^2(\tilde{P})_\pi} = M_{Q|P}(\lambda) $$
pour $\Re\lambda$ suffisamment positif. Cf. \cite[II.1.9]{MW94}. Il définit donc une famille d'opérateurs méromorphes en $\lambda$. On en déduit le résultat suivant.

\begin{lemma}
  Le produit infini
  $$ r_{Q|P}(\pi_\lambda) := \prod_v r_{\tilde{Q}|\tilde{P}}(\pi_{v,\lambda}), \quad \lambda \in \mathfrak{a}_{M,\C}^* $$
  est convergent pour $\Re\lambda$ suffisamment positif et définit une fonction méromorphe en $\lambda$.
\end{lemma}
\begin{remark}
  On conjecture, du moins dans le cas des groupes réductifs, que $r_{Q|P}(\pi_\lambda)$ s'exprime en termes des fonctions $L$ automorphes et des facteurs $\varepsilon$. Nous n'aborderons pas ce point de vue dans cet article.
\end{remark}

\paragraph{Une majoration pour les facteurs normalisants}
Fixons $P \in \mathcal{P}(M)$ et définissons les $(G,M)$-familles suivantes en $\Lambda \in \mathfrak{a}_{M,\C}^*$.
\begin{align*}
  \mathcal{R}_Q(\pi, \Lambda, P) & := R_{Q|P}(\pi)^{-1} R_{Q|P}(\pi_\Lambda), \\
  r_Q(\pi, \Lambda, P) & := r_{Q|P}(\pi)^{-1} r_{Q|P}(\pi_\Lambda).
\end{align*}

\begin{theorem}\label{prop:majoration}
  Soit $L \in \mathcal{L}(M)$. Il existe un entier $n_0$ tel que
  $$ \int_{i(\mathfrak{a}^G_L)^*} |r^S_L(\pi_\lambda, P)| (1+\|\lambda\|)^{-n_0} \dd\lambda < +\infty $$
  pour tout $S \in \mathcal{F}(L)$.
\end{theorem}

C'est le résultat technique principal dans cette section. Effectuons d'abord une réduction. Considérons une fonction de la forme $\lambda \mapsto c(\pi_\lambda)$ où $\lambda \in \R$, on note $\dot{c}(\pi_\lambda) := \left.\frac{\dd}{\dd\mu}\right|_{\mu = 0} c(\pi_{\lambda+\mu})$ pour ne pas le confondre avec le formalisme de $(G,M)$-familles.

\begin{lemma}\label{prop:majoration-alpha}
  Il existe un entier $n_0$ tel que pour tout $\alpha \in \Sigma_P^\text{red}$, on a
  $$ \int_{i(\mathfrak{a}^{M_\alpha}_M)^*} |r_\alpha(\pi_\lambda)^{-1} \dot{r}_\alpha(\pi_\lambda)| (1+\|\lambda\|)^{-n_0} \dd \lambda < +\infty $$
\end{lemma}

\begin{proof}[Démonstration du Théorème \ref{prop:majoration} en supposant le Lemme \ref{prop:majoration-alpha}]
  On observe d'abord que $(r_Q(\pi,\Lambda, P))_Q$ est de la sorte des $(G,M)$-familles considérées dans \cite[\S 7]{Ar82-Eis2} d'après \eqref{eqn:r-prod}. Écrivons $S=L_S U_S$. Alors
  $$ r^S_L(\pi_\lambda, P) = \sum_F \mes(\mathfrak{a}^{L_S}_L / \Z F_L^\vee) \prod_{\alpha \in F} r_\alpha(\pi_\lambda)^{-1} \dot{r}_\alpha(\pi_\lambda) $$
  où $F$ parcourt les sous-ensembles de $\Sigma_M^{L_S, \text{red}}$ tels que $F_L^\vee := F^\vee|_{\mathfrak{a}_L}$ est une base de $\mathfrak{a}^{L_S}_L$. Ainsi, on décompose l'intégrale en question en une somme sur les $F$. Par abus de notation, on note momentanément $\{\varpi_\alpha\}_{\alpha \in F}$ la base duale de $F^\vee|_{\mathfrak{a}_L}$. Pour chaque terme associé à $F$, on utilise la décomposition
  $$ (\mathfrak{a}^G_L)^* = \bigoplus_{\alpha \in F} \R \varpi_\alpha \oplus (\mathfrak{a}^G_{L_S})^* $$
  et le Lemme \ref{prop:majoration-alpha} pour conclure.
\end{proof}

\begin{proof}[Démonstration du Lemme \ref{prop:majoration-alpha}]
  C'est pareil que \cite[Lemma 8.4]{Ar82-Eis2}. Donnons-en une esquisse. Pour simplifier, on peut supposer que $M$ est un Lévi propre maximal standard de $G$. Écrivons $\mathcal{P}(M) = \{P,\bar{P}\}$ où $P$ est standard, $\Sigma_P^\text{red} = \{\alpha\}$. Posons $\varpi := \varpi_\alpha$. Prenons $\chi \in \mathfrak{X}^{\tilde{G}}$ tel que $\mathcal{A}^2(\tilde{P})_{\pi,\chi} \neq \{0\}$.

  Le point de départ est la Proposition \ref{prop:majoration-1}: il existe des entiers $n_0, d_0$ tels que pour tout vecteur $\phi \in \mathcal{A}^2(\tilde{P})_{\chi,\pi}$, il existe une constante $c_\phi$ telle que
  $$ \int_{i(\mathfrak{a}^G_M)^*} |(\Omega^T_{\chi,\pi}(\tilde{P},\lambda)\phi|\phi)| (1+\|\lambda\|)^{-n_0} \dd\lambda \leq c_\phi (1+\|T\|)^{d_0}. $$

  Le Théorème \ref{prop:prod-Eis-1} donne un développement de la forme
  $$ (\Omega^T_{\chi,\pi}(\tilde{P},\lambda)\phi|\phi) = \sum_{X \in \mathcal{E}\text{xp}} q^{T,\tilde{G}}_X(\lambda,\lambda,\phi,\phi) e^{\angles{X,T}}. $$
  Fixons $\phi$ et introduisons un ensemble fini $\mathcal{L}_X \subset \Hom_\R(i\mathfrak{a}_M^*, i\mathfrak{a}_0^*)$ pour tout $X \in \mathcal{E}\text{xp}$, tel qu'il existe des fonctions $\sigma^T_L(\lambda)$, polynomiales en $T$ et analytiques pour $\lambda \in i\mathfrak{a}_M^*$, avec
  $$ q_X^{T,\tilde{G}}(\lambda,\lambda,\phi,\phi) e^{\angles{X,T}} = \sum_{L \in \mathcal{L}_X} \sigma^T_L(\lambda) e^{\angles{L(\lambda)+X,T}}. $$

  Donc il existe un entier $d \gg 0$ tel que, si l'on prend $\xi \in \mathfrak{a}_0$ avec $\angles{X,\xi} < 0$ pour tout $X \in \mathcal{E}\text{xp}$, $X \neq 0$, et pose
  \begin{align*}
    \delta_T & := \text{l'opérateur } (\delta_T \psi)(T') = \psi(T'+\xi) \text{ pour toute $\psi: \mathfrak{a}_0 \to \C$}, \\
    \Delta_T(\lambda) & := \prod_{\substack{X \in \mathcal{E}\text{xp} \setminus \{0\} \\ L \in \mathcal{L}_X}} \left( \delta_T - e^{\angles{L(\lambda)+X, \xi}}\identity \right)^d ,
  \end{align*}
  alors
  $$ \Delta_T(\lambda) (\Omega^T_{\chi,\pi}(\tilde{P},\lambda)\phi|\phi) = \Delta_T(\lambda) (\omega^T_{\chi,\pi}(\tilde{P},\lambda)\phi|\phi). $$

  Puisque $|e^{\angles{L(\lambda)+X,\xi}}| \leq 1$, on en déduit une majoration
  \begin{equation}\label{eqn:ma-omega}
    \int_{i(\mathfrak{a}^G_M)^*} |\Delta_T(\lambda) (\omega^T_{\chi,\pi}(\tilde{P},\lambda)\phi|\phi)| (1+\|\lambda\|)^{-n_0} \dd\lambda \leq c'_\phi (1+\|T\|)^{d_0}
  \end{equation}
  pour une autre constante $c'_\phi$. Dans ce qui suit, nous fixons le vecteur non nul $\phi \in \mathcal{A}^2(\tilde{P})_{\chi,\pi}$.

  Écrivons $\lambda = z\varpi$ où $z \in i\R$. Appliquons la formule fournie par la Proposition \ref{prop:omega-prod-s}. On voit que $(\omega^T_{\chi,\pi}(\tilde{P},\lambda)\phi|\phi))$ est la somme d'au plus deux termes. Il y en a toujours un terme correspondant à $s=1 \in W(M)$, qui est
  $$ \lim_{\lambda' \to \lambda} \sum_{Q \in \{P,\bar{P}\}} (M_{Q|P}(\lambda)^{-1} M_{Q|P}(\lambda')\phi|\phi) \frac{e^{\angles{\lambda'-\lambda,Y_Q(T)}}}{\theta_Q(\lambda'-\lambda)}; $$
  puisque $\dim\mathfrak{a}^G_M = 1$, d'après \cite[\S 7]{Ar82-Eis2} cela est la somme des trois termes
  \begin{align}
    \label{eqn:ma-1} -(R_{\bar{P}|P}(\pi_{z\varpi})^{-1} \dot{R}_{\bar{P}|P}(\pi_{z\varpi})\phi|\phi) \\
    \label{eqn:ma-2} - r_\alpha(\pi_{z\varpi})^{-1} \dot{r}_\alpha(\pi_{z\varpi}) (\phi|\phi) \\
    \label{eqn:ma-3} + 2 \angles{\varpi, T-T_0} (\phi|\phi).
  \end{align}

  L'autre terme éventuel correspond à $s=s_\alpha \in W(M)$, la symétrie orthogonale associée à $\alpha$, si elle existe. De la même façon, c'est égal à
  \begin{equation}\label{eqn:ma-4}
    \mes(\mathfrak{a}^G_M/\Z\alpha^\vee) \left( \frac{(M_{P|P}(s_\alpha, z\varpi)^{-1}\phi|\phi)e^{2z\angles{\varpi, T}}}{2z} + \frac{(M_{P|P}(s_\alpha,z\varpi)\phi|\phi)e^{-2z\angles{\varpi,T}}}{-2z} \right)
  \end{equation}
  lorsque $z \neq 0$.

  Montrons que le terme \eqref{eqn:ma-1} satisfait à une majoration de la forme
  \begin{equation}\label{eqn:ma-cherchee}
    \int_{i\R} |\Delta_T(z\varpi)(\cdots)|(1+|z|)^{-n_0} \dd z \leq c''_\phi (1+\|T\|)^{d_0}
  \end{equation}
  pour un entier $n_0$ suffisamment grand et une constante $c''_\phi$. Comme ce que nous avons constaté dans la définition de $R_{\bar{P}|P}(\pi_{z\varpi})$, c'est essentiellement un problème local. La propriété (\textbf{R8}) dans \cite[\S 3.1]{Li11a} entraîne que cela est vrai sans $\Delta_T(z\varpi)$. Afin d'établir \eqref{eqn:ma-cherchee}, il suffit de noter que l'opérateur $\Delta_T(z\varpi)$ a pour effet d'introduire le facteur uniformément borné en $z$:
  $$ \prod_{\substack{X \in \mathcal{E}\text{xp} \setminus \{0\} \\ L \in \mathcal{L}_X}} \left( 1-e^{\angles{L(z\varpi)+X,\xi}} \right). $$

  Montrons ensuite que \eqref{eqn:ma-3} satisfait aussi à \eqref{eqn:ma-cherchee}. Il suffit de constater que l'effet $\Delta_T(\lambda)$ est de le multiplier par
  $$ \prod_{\substack{X \in \mathcal{E}\text{xp} \setminus \{0\} \\ L \in \mathcal{L}_X}} \left( 2\angles{\varpi,\xi} - e^{\angles{L(z\varpi)+X,\xi}} \right) $$
  qui est toujours uniformément borné en $z$.

  Supposons que $s_\alpha \in W(M)$ existe et considérons \eqref{eqn:ma-4}. Il est borné pour $z \in i\R$ éloigné de $0$ car $M_{P|P}(s_\alpha, z\varpi)$ est un opérateur unitaire. D'autre part, la somme des deux termes dans \eqref{eqn:ma-4} est régulière en $z=0$, donc est bornée pour $z$ proche de $0$. L'effet de $\Delta_T(z\varpi)$ sur \eqref{eqn:ma-4} est d'introduire le facteur
  $$ \prod_{\substack{X \in \mathcal{E}\text{xp} \setminus \{0\} \\ L \in \mathcal{L}_X}} \left( e^{\pm 2z \angles{\varpi,\xi}}-e^{\angles{L(z\varpi)+X,\xi}} \right) $$
  pour les termes avec $\pm 2z$, respectivement. Ceci est lisse en $z$ et uniformément borné. D'où la majoration \eqref{eqn:ma-cherchee}.

  La majoration \eqref{eqn:ma-omega} et les majorations de la forme \eqref{eqn:ma-cherchee} satisfaites par \eqref{eqn:ma-1}, \eqref{eqn:ma-3}, \eqref{eqn:ma-4} entraînent qu'il existe une constante $c''$ et un entier $n_0$ tels que
  $$ \int_{i\R} |\Delta_T(z\varpi) r_\alpha(\pi_{z\varpi})^{-1} \dot{r}_\alpha(\pi_{z\varpi})| (1+|z|)^{-n_0} \dd z \leq  c'' (1+\|T\|)^{d_0}. $$

  Or $r_\alpha(\pi_{z\varpi})^{-1} \dot{r}_\alpha(\pi_{z\varpi})$ étant indépendant de $T$, l'effet de $\Delta_T(z\varpi)$ est d'introduire le facteur $\prod_{X,L} (1-e^{\angles{L(z\varpi)+X,\xi}})$, qui est éloigné de $0$. On en déduit la même majoration sans $\Delta_T(z\varpi)$ quitte à agrandir $c''$, ce qu'il faut démontrer.
\end{proof}
\begin{remark}
  Il serait plus satisfaisant d'établir une majoration uniforme en $\pi$, comme ce qu'est fait dans \cite[Theorem 5.3]{Mu02}. Une telle amélioration sera utile pour des problèmes de convergence absolue du côté spectral, cf. \cite{FLM11}.
\end{remark}

\paragraph{Le développement spectral fin}
Fixons maintenant des données suivantes comme au \S\ref{sec:formule-explicite}.
\begin{itemize}
  \item $f \in \mathcal{H}(\tilde{G}^1)$,
  \item $\chi \in \mathfrak{X}^{\tilde{G}}$,
  \item $M \in \mathcal{L}(M_0)$, $P \in \mathcal{P}(M)$,
  \item $L \in \mathcal{L}(M)$,
  \item $s \in W^L(M)_\text{reg}$,
\end{itemize}

Soit $\pi \in \Pi_\mathrm{unit}(\tilde{M}^1)$. Nous avons déjà défini la $(G,M)$-famille $\mathcal{R}_Q(\pi_\lambda, P)$ (où $Q \in \mathcal{P}(M)$), d'où les opérateurs $\mathcal{R}'_S(\pi_\lambda, P)$ pour chaque $S \in \mathcal{F}(L)$. Tout d'abord, il convient d'introduire la norme de traces $\|\cdot\|_1$ pour les opérateurs.

\begin{definition}\label{def:trace}
  Soient $(H, (\cdot, \cdot))$ un espace d'Hilbert et $A: H \to H$ un opérateur borné. On note
  $$ |A| := (A^* A)^{\frac{1}{2}} $$
  l'opérateur obtenu par le calcul fonctionnel. Soit $\mathcal{B}$ une base orthonormée de $H$, on définit la norme de traces
  $$ \|A\|_1 := \sum_{v \in \mathcal{B}} (|A|v , v) \quad \in \R \cup \{+\infty\}. $$
  C'est connu que $\|A\|_1$ ne dépend pas de $\mathcal{B}$. Un opérateur borné $A: H \to H$ est dit à trace si $\|A\|_1 < +\infty$; dans ce cas-là on peut bien définir sa trace
  $$ \Tr A := \sum_{v \in \mathcal{B}} (Av,v) $$
  qui est indépendante de la base orthonormée $\mathcal{B}$.
\end{definition}

Ces notions sont standards en la théorie des opérateurs; voir par exemple \cite[\S 18]{Co00}. Ci-dessous sont les propriétés de $\|\cdot\|_1$ qui nous concernent.
\begin{enumerate}
  \item Les opérateurs à trace forment un idéal de l'algèbre des opérateurs bornés, pour lequel $\| \cdot \|_1$ est une norme.
  \item Si $A$ est à trace, alors $|\Tr A| \leq \|A\|_1$.
  \item Si $A$ est à trace et $B$ est borné, alors $\Tr(AB)=\Tr(BA)$.
  \item Soit $U: H \to H$ un opérateur unitaire, alors $\|UA\|_1 = \|AU\|_1 = \|A\|_1$.
\end{enumerate}
Les opérateurs que nous considérerons sont tous de rang fini, donc ces notions appartiennent effectivement à l'algèbre linéaire élémentaire.

\begin{lemma}\label{prop:majoration-RI}
  Pour tout entier $n_0 \geq 0$ et toute $\pi \in \Pi_\mathrm{unit}(\tilde{M}^1)$, il existe une constante $c_{\pi}(n_0)$ telle que pour tous $\lambda \in i(\mathfrak{a}^G_L)^*$, on a
  $$ \| \mathcal{R}'_S(\pi_\lambda,P) \mathcal{I}_{\tilde{P}}(\lambda,f)_{\chi,\pi} \|_1 \leq c_\pi(n_0) (1+\|\lambda\|)^{-n_0}. $$
\end{lemma}
\begin{proof}
  Vu la $\tilde{K}$-finitude de $f$, la trace et la norme $\|\cdot\|_1$ se calculent dans un espace de dimension finie. D'après les définitions des opérateurs $\mathcal{R}_Q(\pi_\lambda, P)$ et $\mathcal{I}_{\tilde{P}}(\lambda,f)_{\chi,\pi}$, on se ramène à une situation locale. Les places non archimédiennes ne posent aucune difficulté: elles donnent des facteurs uniformément bornés en $\lambda$. En les places archimédiennes, on conclut par les propriétés suivantes.
  \begin{enumerate}
    \item Les coefficients $\tilde{K}_\infty$-finis des opérateurs d'entrelacement normalisés, ainsi que leurs dérivés, sont à croissance modérée en $\lambda \in i\mathfrak{a}_M^*$ (voir \cite[\S 3.1]{Li11a}).
    \item Les coefficients de l'opérateur $\mathcal{I}_{\tilde{P}}(\lambda,f)_{\chi,\pi}$ en les places archimédiennes sont à décroissance rapide en $\lambda \in i\mathfrak{a}_M^*$. En effet, ceci est la partie facile du théorème de Paley-Wiener d'Arthur \cite{Ar83} qui vaut pour les revêtements.
  \end{enumerate}
\end{proof}

\begin{lemma}\label{prop:conv-abs}
  L'intégrale double
  $$ \sum_{\pi \in \Pi_\mathrm{unit}(\tilde{M}^1)}\; \int_{i(\mathfrak{a}^G_L)^*} \Tr(\mathcal{M}_L(\tilde{P},\lambda) M_{P|P}(s,0) \mathcal{I}_{\tilde{P}}(\lambda,f)_{\chi,\pi}) \dd\lambda $$
  converge absolument.
\end{lemma}
\begin{proof}
  Puisque l'opérateur d'entrelacement $M_{P|P}(s,0)$ est unitaire, on a
  \begin{align*}
    |\Tr(\mathcal{M}_L(\tilde{P},\lambda) M_{P|P}(s,0) \mathcal{I}_{\tilde{P}}(\lambda,f)_{\chi,\pi})| & \leq \|\mathcal{M}_L(\tilde{P},\lambda) M_{P|P}(s,0) \mathcal{I}_{\tilde{P}}(\lambda,f)_{\chi,\pi}\|_1 \\
    & = \|\mathcal{M}_L(\tilde{P},\lambda) \mathcal{I}_{\tilde{P}}(\lambda,f)_{\chi,\pi} M_{P|P}(s,0)\|_1 \\
    & = \|\mathcal{M}_L(\tilde{P},\lambda) \mathcal{I}_{\tilde{P}}(\lambda,f)_{\chi,\pi} \|_1.
  \end{align*}
  Donc il suffit de montrer la convergence absolue de
  $$ \sum_{\pi \in \Pi_\mathrm{unit}(\tilde{M}^1)}\; \int_{i(\mathfrak{a}^G_L)^*} \|\mathcal{M}_L(\tilde{P},\lambda) \mathcal{I}_{\tilde{P}}(\lambda,f)_{\chi,\pi} \|_1 \dd\lambda. $$

  En appliquant la formule de descente pour $(G,L)$-familles \cite[Lemme 4.2.4]{Li10a}, $\mathcal{M}_L(\tilde{P},\lambda) \mathcal{I}_{\tilde{P}}(\lambda,f)_{\chi,\pi}$ est égal à
  $$ \sum_{S \in \mathcal{F}(L)} \mathcal{R}'_S(\pi_\lambda, P) r^S_L(\pi_\lambda,P) \mathcal{I}_{\tilde{P}}(\lambda,f)_{\chi,\pi}. $$
  D'où
  \begin{multline*}
    \sum_{\pi} \int_{i(\mathfrak{a}^G_L)^*} \|\mathcal{M}_L(\tilde{P},\lambda) \mathcal{I}_{\tilde{P}}(\lambda,f)_{\chi,\pi} \|_1 \dd\lambda \leq \\
    \sum_{\pi} \sum_{S \in \mathcal{F}(L)} \int_{i(\mathfrak{a}^G_L)^*} \|\mathcal{R}'_S(\pi_\lambda, P)  \mathcal{I}_{\tilde{P}}(\lambda,f)_{\chi,\pi}\|_1 \cdot |r^S_L(\pi_\lambda, P)| \dd\lambda .
  \end{multline*}

  Puisque la somme sur $\pi$ est finie, on conclut en appliquant le Lemme \ref{prop:majoration-RI} et le Théorème \ref{prop:majoration}.
\end{proof}

Rappelons que les objets $\mathcal{M}_L(\tilde{P},\lambda)$, $M_{P|P}(s,0)$ et $\mathcal{I}_{\tilde{P}}(\lambda,f)_{\chi,\pi}$ sont définis pour tout $M \in \mathcal{L}(M_0)$ et tout $P \in \mathcal{F}(M_0)$ pas forcément standard.

\begin{theorem}[Cf. {\cite[Theorem 8.2]{Ar82-Eis2}}]\label{prop:formule-J_chi}
  On a
  \begin{multline*}
    J_\chi(f) = \sum_{M \in \mathcal{L}(M_0)} \sum_{\pi \in \Pi_\mathrm{unit}(\tilde{M}^1)} \sum_{L \in \mathcal{L}(M)} \sum_{s \in W^L(M)_\text{reg}} |W^M_0| |W^G_0|^{-1} |\mathcal{P}(M)|^{-1} \cdot \\
    \cdot |\det(s-1|\mathfrak{a}^L_M)|^{-1} \int_{i(\mathfrak{a}^G_L)^*} \sum_{P \in \mathcal{P}(M)} \Tr(\mathcal{M}_L(\tilde{P},\lambda) M_{P|P}(s,0) \mathcal{I}_{\tilde{P}}(\lambda,f)_{\chi,\pi}) \dd\lambda .
  \end{multline*}
\end{theorem}
\begin{proof}
  Vu le Théorème \ref{prop:P^T(B)-J^T}, on prend $B \in C_c^\infty(i\mathfrak{h}^*/i\mathfrak{a}_G^*)^W$ avec $B(0)=1$ et on calcule $\lim_{\epsilon \to 0} P^{T_0}(B^\epsilon)$ à l'aide du Théorème \ref{prop:formule-explicite}, qui dit que $P^{T_0}(B^\epsilon)$ est égal à
  $$ \sum_{\substack{P=MU \\ P \supset P_0}} \sum_{\pi, L, s} |\mathcal{P}(M)|^{-1} |\det(s-1|\mathfrak{a}^L_M)|^{-1} \int_{i(\mathfrak{a}^G_L)^*} \Tr(\mathcal{M}^{T_0}_L(\tilde{P},\lambda) M_{P|P}(s,0) \mathcal{I}_{\tilde{P}}(\lambda,f)_{\chi,\pi}) (B^\epsilon)_\pi(\lambda) \dd\lambda .$$

  D'abord, $\mathcal{M}^{T_0}_L(\tilde{P}, \lambda)$ est la limite
  $$ \lim_{\Lambda \to 0} \sum_{Q_1 \in \mathcal{P}(L)} e^{\angles{\Lambda, Y_{Q_1}(T_0)}} \mathcal{M}_{Q_1}(\tilde{P}, \lambda, \Lambda) \theta_{Q_1}(\Lambda)^{-1}. $$
  Or $Y_{Q_1}(T_0)$ est la projection de $T_0$ sur $\mathfrak{a}_L$, qui ne dépend pas de $Q_1$, donc $\mathcal{M}^{T_0}_L(\lambda, \tilde{P}) = \mathcal{M}_L(\lambda, \tilde{P})$.

  On remplace ensuite la somme sur $P \supset P_0$ par une somme sur $M \in \mathcal{L}(M_0)$ et $P \in \mathcal{P}(M)$ en introduisant simultanément le facteur $|W^M_0| |W^G_0|^{-1}$. Cela donne l'expression suivante pour $P^{T_0}(B^\epsilon)$
  \begin{multline*}
    \sum_{M,\pi,L, s} |W^M_0| |W^G_0|^{-1} |\mathcal{P}(M)|^{-1} |\det(s-1|\mathfrak{a}^L_M)|^{-1} \\ \int_{i(\mathfrak{a}^G_L)^*} \sum_{P \in \mathcal{P}(M)} \Tr(\mathcal{M}_L(\tilde{P},\lambda) M_{P|P}(s,0) \mathcal{I}_{\tilde{P}}(\lambda,f)_{\chi,\pi}) (B^\epsilon)_\pi(\lambda) \dd\lambda.
  \end{multline*}

  On a $(B^\epsilon)_\pi(\lambda) = B(\epsilon(iY_\pi+\lambda))$. Il est borné par $\sup_{\lambda \in \mathfrak{h}^*} |B(\lambda)|$ et converge simplement vers $B(0)=1$ lorsque $\epsilon \to 0$. Le théorème de convergence dominée et le Lemme \ref{prop:conv-abs} permettent de conclure.
\end{proof}

\begin{corollary}\label{prop:formule-J_chi-simple}
  La formule dans le Théorème \ref{prop:formule-J_chi} peut aussi s'écrire comme
  \begin{multline*}
    J_\chi(f) = \sum_{M \in \mathcal{L}(M_0)} \sum_{\pi \in \Pi_\mathrm{unit}(\tilde{M}^1)} \sum_{L \in \mathcal{L}(M)} \sum_{s \in W^L(M)_\text{reg}} |W^M_0| |W^G_0|^{-1} \cdot \\
    \cdot |\det(s-1|\mathfrak{a}^L_M)|^{-1} \int_{i(\mathfrak{a}^G_L)^*} \Tr(\mathcal{M}_L(\tilde{P},\lambda) M_{P|P}(s,0) \mathcal{I}_{\tilde{P}}(\lambda,f)_{\chi,\pi}) \dd\lambda
  \end{multline*}
  où $P \in \mathcal{P}(M)$ est arbitraire.
\end{corollary}
\begin{proof}
  Montrons que le terme $\Tr(\cdots)$ est indépendant de $P$. Soient $P, P_1 \in \mathcal{P}(M)$. Si $\lambda, \Lambda \in i\mathfrak{a}_M^*$, alors
  \begin{align*}
    \mathcal{M}_Q(\tilde{P}_1, \lambda, \Lambda) & = M_{P|P_1}(\lambda)^{-1} M_{Q|P}(\lambda)^{-1} M_{Q|P}(\lambda+\Lambda) M_{P|P_1}(\lambda+\Lambda) \\
    & = M_{P|P_1}(\lambda)^{-1} \mathcal{M}_Q(\tilde{P}, \lambda, \Lambda) M_{P|P_1}(\lambda+\Lambda)
  \end{align*}
  pour tout $Q \in \mathcal{P}(M)$. Fixons $\lambda$, alors la $(G,L)$-famille en $\Lambda$ induite vérifie la même propriété. On en déduit
  $$ \mathcal{M}_L(\tilde{P}_1, \lambda) = M_{P|P_1}(\lambda)^{-1} \mathcal{M}_L(\tilde{P}, \lambda) M_{P|P_1}(\lambda). $$

  Supposons maintenant que $\lambda \in i\mathfrak{a}_L^*$ et $s \in W^L(M)_\text{reg}$. On a $M_{P|P}(s,0)=M_{P|P}(s,\lambda)$,  $M_{P_1|P_1}(s,0)=M_{P_1|P_1}(s,\lambda)$. En utilisant les propriétés des opérateurs d'entrelacement globaux, on obtient
  \begin{align*}
    \mathcal{M}_L(\tilde{P}_1,\lambda) M_{P_1|P_1}(s,0) \mathcal{I}_{\tilde{P}_1}(\lambda,f)_{\chi,\pi} & = M_{P|P_1}(\lambda)^{-1} \mathcal{M}_L(\tilde{P}, \lambda) M_{P|P_1}(\lambda) M_{P_1|P_1}(s,0) \mathcal{I}_{\tilde{P}_1}(\lambda,f)_{\chi,\pi} \\
    & = M_{P|P_1}(\lambda)^{-1} \mathcal{M}_L(\tilde{P}, \lambda) M_{P|P}(s,0) M_{P|P_1}(\lambda) \mathcal{I}_{\tilde{P}_1}(\lambda,f)_{\chi,\pi} \\
    & = M_{P|P_1}(\lambda)^{-1} \mathcal{M}_L(\tilde{P},\lambda) M_{P|P}(s,0) \mathcal{I}_{\tilde{P}}(\lambda,f)_{\chi,\pi} M_{P|P_1}(\lambda).
  \end{align*}
  Donc cet opérateur a la même trace que $\mathcal{M}_L(\tilde{P},\lambda) M_{P|P}(s,0) \mathcal{I}_{\tilde{P}}(\lambda,f)_{\chi,\pi}$.
\end{proof}

\section{Coefficients discrets}\label{sec:coef}
Soient $t > 0$ et $M \in \mathcal{L}(M_0)$, on définit
$$ \Pi_t(\tilde{M}^1) := \{\pi \in \Pi_\mathrm{unit}(\tilde{M}^1) : \|\Im \nu_\pi\|=t \} $$
où $\nu_\pi \in \mathfrak{h}^*_\C/W^M$ est défini dans \S\ref{sec:asymptotique}; rappelons que le caractère infinitésimal $\nu_\pi$ est bien défini comme un $W^M$-orbite dans $\mathfrak{h}^*/i\mathfrak{a}_M^*$ et on en prend un représentant de la norme minimale.

On définit
$$ J_t(f) := \sum_{\chi: \|\Im\nu_\chi\|=t} J_\chi(f), \quad f \in \mathcal{H}(\tilde{G}^1), $$
alors on a $\sum_{t \geq 0} J_t(f) = J(f)$. A priori, $J_t(f)$ est défini par une somme absolument convergente. Or il résultera du Lemme \ref{prop:carinf} qu'elle est effectivement une somme finie pourvu que l'on fixe les $\tilde{K}$-types de $f$. %FIXME: énoncer-la comme un corollaire à la fin.

Pour $t$ fixé, on s'intéresse à la partie dite $t$-discrète de $J(f)$ définie comme ci-dessous. Notons $L^2_{\text{disc},t}(M(F) \backslash \tilde{M}^1)$ la somme directe des composantes $\pi$-isotypiques de $L^2_\text{disc}(M(F) \backslash \tilde{M}^1)$ avec $\pi \in \Pi_t(\tilde{M}^1)$. Pour $P \in \mathcal{P}(M)$, on définit la représentation $\mathcal{I}_{\tilde{P},\text{disc},t}(\lambda)$ comme l'induite parabolique normalisée de $L^2_{\text{disc},t}(M(F) \backslash \tilde{M}^1) \otimes e^{\angles{\lambda, H_M(\cdot)}}$. Ces sous-représentations de $\mathcal{I}_{\tilde{P}}(\lambda)$ sont respectées par les opérateurs d'entrelacement $M_{Q|P}(w,\lambda)$ pour tous $P,Q \in \mathcal{P}(M)$, $\lambda \in \mathfrak{a}_{M,\C}^*$ et $w \in W^G_0$.

\begin{proposition}[Cf. {\cite[Lemma 4.1]{Ar88-TF2}}]\label{prop:adm}
  Pour tout $t \geq 0$, la représentation $\mathcal{I}_{\tilde{P},\mathrm{disc},t}(\lambda)$ est admissible.
\end{proposition}
Rappelons que l'admissibilité signifie que chaque $\tilde{K}$-type est de multiplicité finie.
\begin{proof}
  Comme dans le cas local, l'admissibilité est préservée par induction parabolique, donc on se ramène à prouver l'admissibilité de $L^2_{\text{disc},t}(G(F) \backslash \tilde{G}^1)$. On décompose
  $$ L^2_{\text{disc},t}(G(F) \backslash \tilde{G}^1) = \bigoplus_{\chi : \|\Im \nu_\chi\|=t} L^2_{\text{disc},\chi}(G(F) \backslash \tilde{G}^1). $$

  La construction du spectre discret entraîne que chaque $L^2_{\text{disc},\chi}(G(F) \backslash \tilde{G}^1)$ est admissible (d'après les mêmes références que dans la preuve de la Proposition \ref{prop:finitude-pi}). D'après la construction du spectre discret \cite[V]{MW94}, le lemme suivant permet de conclure.
\end{proof}

\begin{lemma}\label{prop:carinf}
  Soit $\Gamma$ un sous-ensemble fini de $\Pi_\mathrm{unit}(\tilde{K})$. Désignons par $\Xi$ l'ensemble des caractères infinitésimaux $\nu_\pi \in \mathfrak{h}^*_\C/W$ des représentations automorphes cuspidales $\pi \in \Pi_\mathrm{unit}(\tilde{G}^1)$ contenant des $\tilde{K}$-types dans $\Gamma$. Alors
  \begin{enumerate}
    \item $\Xi$ est discret dans $\mathfrak{h}^*_\C/W \mod i\mathfrak{a}_G^*$,
    \item l'image réciproque de $\Xi$ dans $\mathfrak{h}_\C^*$ a pour partie réelle bornée.
  \end{enumerate}
  Idem pour les sous-groupes de Lévi de $\tilde{G}$.
\end{lemma}
\begin{proof}
  Montrons d'abord que $\Xi$ est discret. En effet, l'élément de Casimir dans $U(\mathfrak{g}_{\infty,\C})$ agit sur l'espace des formes automorphes cuspidales sur $G(F) \backslash \tilde{G}^1$. Il suffit de montrer que le spectre de l'élément de Casimir est discret et chaque valeur propre est de multiplicité finie. Ceci est bien connu; voir par exemple \cite[Theorem 9.1]{Do82}, qui s'applique également aux revêtements. % ou bien \cite[Theorem 7.6]{OW81}

  La deuxième assertion concerne les représentations unitaires irréductibles de $\tilde{G}_\infty$ contenant un $\tilde{K}_\infty$-type prescrit. Il n'existe qu'un nombre fini de $i\mathfrak{a}^*_{\tilde{G}_\infty}$-orbites de représentations de carré intégrable modulo le centre qui contiennent le $\tilde{K}_\infty$-type choisi (voir \cite[7.7.3]{Wall88}), donc l'assertion vaut pour de telles représentations. On en déduit le cas des représentations tempérées car l'induction parabolique unitaire ne touche que la partie imaginaire du caractère infinitésimal.

  Considérons le cas général où la représentation en question $\pi_\infty \in \Pi_\mathrm{unit}(\tilde{G}_\infty)$ est le quotient de Langlands
  $$ J_{\tilde{\bar{Q}}|\tilde{Q}}(\sigma_\lambda): \mathcal{I}_{\tilde{Q}}(\sigma_\lambda) \twoheadrightarrow \pi_\infty $$
  où
  \begin{itemize}
    \item $\tilde{Q} = \tilde{L} U$ est un parabolique standard de $\tilde{G}_\infty$ (on fixe un parabolique minimal de $\tilde{G}_\infty$),
    \item $\sigma$ est une représentation tempérée de $\tilde{L}$,
    \item $\lambda \in \mathfrak{a}^*_{L,\C}$ est tel que
    \begin{equation}\label{eqn:lambda-positif-1}
      \angles{\Re\lambda, \alpha^\vee} > 0, \quad \alpha \in \Delta_Q .
    \end{equation}
  \end{itemize}
  % citer d'ailleurs Borel et Wallach, IV. 4.4.
  Un lemme de Langlands \cite[5.3.4]{Wall88} dit que, pour tous $v \in \mathcal{I}_{\tilde{Q}}(\sigma_\lambda)$ et $\check{v} \in \mathcal{I}_{\tilde{Q}}(\sigma^\vee_{-\lambda})$, on a
  \begin{equation}\label{eqn:Langlands}
    \lim_{a \xrightarrow{Q} \infty} e^{-\angles{\lambda-\rho_Q, H_L(a)}} \angles{\check{v}, \mathcal{I}_{\tilde{Q}}(\sigma_\lambda, a)v} = \angles{\check{v}(1), J_{\tilde{\bar{Q}}|\tilde{Q}}(\sigma_\lambda)v(1) }_{\sigma}.
  \end{equation}
  Ici la notation signifie que $a \in \widetilde{A_L}^\circ$, la composante neutre de la composante déployée de $Z_{\tilde{L}}$, et $H_L(a)$ tend fortement vers l'infini dans la chambre
  $$ \{ H \in \mathfrak{a}_L : \forall \alpha \in \Delta_Q, \; \angles{\alpha,H} > 0 \}, $$
  cf. la Définition \ref{def:limite-T}.

  On prend $v, \check{v}$ de sorte que $\angles{\check{v}, \mathcal{I}_{\tilde{Q}}(\sigma_\lambda, \cdot)v}$ est un coefficient matriciel de $\pi_\infty$, qui est borné puisque $\pi_\infty$ est unitaire, et que le côté à droite de \eqref{eqn:Langlands} est non nul. Il en résulte que
  \begin{equation}\label{eqn:lambda-positif-2}
    \angles{\Re\lambda - \rho_Q, \varpi_\alpha^\vee} \leq 0, \quad \alpha \in \Delta_Q .
  \end{equation}

  Vu \eqref{eqn:lambda-positif-1} et \eqref{eqn:lambda-positif-2}, c'est bien connu que $\Re\lambda$ appartient à un sous-ensemble borné de $\mathfrak{a}_L^*$. Puisque $\nu_{\pi_\infty} = \nu_{\sigma} + \lambda$ modulo $W$, l'assertion en résulte.
\end{proof}

\begin{corollary}
  Soient $f \in \mathcal{H}(\tilde{G}^1)$, $P=MU \in \mathcal{F}(M_0)$, $s \in W^G(M)_{\mathrm{reg}}$, alors l'opérateur $M_{P|P}(s,0) \mathcal{I}_{\tilde{P},\mathrm{disc},t}(0,f)$ est à trace.
\end{corollary}

Les termes avec $L=G$ dans le Corollaire \ref{prop:formule-J_chi-simple}, sommés sur tout $\chi$ avec $\|\Im\nu_\chi\|=t$, définissent une nouvelle distribution, à savoir la partie $t$-discrète de $J(f)$:
$$ J_{\text{disc},t}(f) := \sum_{M \in \mathcal{L}(M_0)} \sum_{s \in W^G(M)_\text{reg}} |W^M_0| |W^G_0|^{-1} |\det(1-s|\mathfrak{a}^G_M)|^{-1} \Tr(M_{P|P}(s,0) \mathcal{I}_{\tilde{P},\mathrm{disc},t}(0,f)) $$
où $P \in \mathcal{P}(M)$ est arbitraire.

Vu ladite admissibilité, on peut définir des coefficients uniques $a^{\tilde{G}}_\text{disc}(\pi) \in \C$, pour tout $\pi \in \Pi_t(\tilde{G}^1)$, tels que
$$ J_{\text{disc},t}(f) = \sum_{\pi \in \Pi_t(\tilde{G}^1)} a^{\tilde{G}}_\text{disc}(\pi) \Tr\pi(f), \quad f \in \mathcal{H}(\tilde{G}^1). $$

Il conviendra de définir un sous-ensemble de $\Pi_t(\tilde{G}^1)$ qui sera un domaine plus commode des coefficients $a^{\tilde{G}}_\text{disc}(\cdot)$ dans les articles suivants. Notons $\Pi_{\text{disc},t}(\tilde{G}^1)$ l'ensemble des constituants irréductibles (restreintes à $\tilde{G}^1$) des induites paraboliques normalisées de $\sigma \in \Pi_t(\tilde{M}^1)$ tel que $a^{\tilde{M}}_\text{disc}(\sigma) \neq 0$. C'est clair que si $\pi \in \Pi_t(\tilde{G}^1)$ et $a^{\tilde{G}}_\text{disc}(\pi) \neq 0$, alors $\pi \in \Pi_{\text{disc},t}(\tilde{G}^1)$.

\begin{proposition}
  Soit $\Gamma$ un sous-ensemble fini de $\Pi_\mathrm{unit}(\tilde{K})$. Alors il n'existe qu'un nombre fini de $\pi \in \Pi_{\mathrm{disc},t}(\tilde{G})$ dont la restriction à $\tilde{K}$ contient un élément de $\Gamma$.
\end{proposition}
\begin{proof}
  Observons que les éléments de $\Pi_{\text{disc},t}(\tilde{G})$ sont des constituants irréductibles de $\mathcal{I}_{\tilde{P},\text{disc},t}(0)$, où $P$ parcourt $\mathcal{F}(M_0)$. La finitude découle donc de l'admissibilité.
\end{proof}

\begin{remark}
  Jusqu'à maintenant nous avons fixé $t$ et nous n'étudions que $J_{\text{disc},t}(f)$. Il serait tentant de définir $J_\text{disc}(f)$ comme la somme des termes correspondants à $L=G$  dans le Corollaire \ref{prop:formule-J_chi-simple} et essayer de montrer que $J_\text{disc}(f) = \sum_t J_{\text{disc},t}(f)$. Il s'agit d'un problème de la convergence absolue de $\sum_{t \geq 0} J_{\text{disc},t}(f)$ comme une intégrale double. Nous ne traitons pas ce problème ici, cependant il convient de remarquer que le cas des groupes réductifs connexes est résolu dans \cite{FLM11}.
\end{remark}

\bibliographystyle{abbrv-fr}
\bibliography{metaplectic}

\bigskip
\begin{flushleft}
  Wen-Wei Li \\
  Institut de Mathématiques de Jussieu \\
  175 rue du Chevaleret, 75013 Paris  \\
  France \\
  Adresse électronique: \texttt{wenweili@math.jussieu.fr}
\end{flushleft}

%\printindex[iFT3]
\end{document}